\documentclass[a4paper,11pt,reqno]{amsart}
\usepackage{amsmath}
\usepackage{amsthm,enumerate}

\usepackage{graphicx}
\usepackage{amssymb}

\usepackage{appendix}

\usepackage[italian,english]{babel}

\usepackage[utf8x]{inputenc}
\usepackage{fancyhdr}

\usepackage{calc}
\usepackage{url}

\usepackage[text={6.3in,8.6in},centering]{geometry}

\usepackage{srcltx, inputenc}

\usepackage[T1]{fontenc}

\usepackage[usenames,dvipsnames]{color}

\makeatletter
\def\cleardoublepage{\clearpage\if@twoside \ifodd\c@page\else%
         \hbox{}%
     \thispagestyle{empty}
     \newpage%
     \if@twocolumn\hbox{}\newpage\fi\fi\fi}
\makeatother

\theoremstyle{plain}
\newtheorem{theorem}{Theorem}[section]

\newtheorem{definition}[theorem]{Definition}

\newtheorem{lemma}[theorem]{Lemma}

\newtheorem{proposition}[theorem]{Proposition}
\newtheorem{remark}[theorem]{Remark}

\numberwithin{equation}{section}
\theoremstyle{definition}

\newcommand{\R}{\ensuremath{\mathbb{R}}}

\begin{document}

\title[]{Global existence of solutions and smoothing effects\\ for classes of reaction-diffusion equations on manifolds}

\author{Gabriele Grillo}
\address{\hbox{\parbox{5.7in}{\medskip\noindent{Dipartimento di Matematica,\\
Politecnico di Milano,\\
   Piazza Leonardo da Vinci 32, 20133 Milano, Italy.
   \\[3pt]
        \em{E-mail address: }{\tt
          gabriele.grillo@polimi.it
          }}}}}

\author{Giulia Meglioli}
\address{\hbox{\parbox{5.7in}{\medskip\noindent{Dipartimento di Matematica,\\
Politecnico di Milano,\\
   Piazza Leonardo da Vinci 32, 20133 Milano, Italy.
   \\[3pt]
        \em{E-mail address: }{\tt
          giulia.meglioli@polimi.it
          }}}}}

\author{Fabio Punzo}
\address{\hbox{\parbox{5.7in}{\medskip\noindent{Dipartimento di Matematica,\\
Politecnico di Milano,\\
   Piazza Leonardo da Vinci 32, 20133 Milano, Italy. \\[3pt]
        \em{E-mail address: }{\tt
          fabio.punzo@polimi.it}}}}}

\keywords{Reaction diffusion equations. Riemannian manifolds. Blow-up. Global existence. Diffusions with weights.}

\subjclass[2010]{Primary: 35K57. Secondary: 35B44, 58J35, 35K65, 35R01.}

\maketitle

\maketitle              

\begin{abstract} We consider the porous medium equation with a power-like reaction term, posed on Riemannian manifolds. Under certain assumptions on $p$ and $m$ in \eqref{problema}, and for small enough nonnegative initial data, we prove existence of global in time solutions, provided that the Sobolev inequality holds on the manifold. Furthermore, when both the Sobolev and the Poincaré inequality hold, similar results hold under weaker assumptions on the forcing term. By the same functional analytic methods, we investigate global existence for solutions to the porous medium equation with source term and variable density in ${\mathbb R}^n$.
\end{abstract}

\bigskip
\bigskip

\section{Introduction}

We investigate existence of global in time solutions to nonlinear reaction-diffusion problems of the following type:
\begin{equation}\label{problema}
\begin{cases}
\, u_t= \Delta u^m +\, u^p & \text{in}\,\, M\times (0,T) \\
\,\; u =u_0 &\text{in}\,\, M\times \{0\}\,,
\end{cases}
\end{equation}
where $M$ is an $N-$dimensional complete noncompact Riemannian manifold of infinite volume, $\Delta$ being the Laplace-Beltrami operator on $M$ and $T\in (0,\infty]$. We shall assume throughout this paper that $$ N\geq 3,\quad \quad m\,>\,1,\quad \quad  p\,>\,m,$$ so that we are concerned with the case of  \it degenerate diffusions \rm of porous medium type (see \cite{V}), and that the initial datum $u_0$ is nonnegative.

Let L$^q(M)$ be the space of those measurable functions $f$ such that $|f|^q$ is integrable w.r.t. the Riemannian measure $\mu$. We shall always assume that $M$ supports the Sobolev inequality, namely that:
\begin{equation}\label{S}
(\textrm{Sobolev\ inequality)}\ \ \ \ \ \ \|v\|_{L^{2^*}(M)} \le \frac{1}{C_s} \|\nabla v\|_{L^2(M)}\quad \text{for any}\,\,\, v\in C_c^{\infty}(M),
\end{equation}
where $C_s$ is a positive constant and $2^*:=\frac{2N}{N-2}$. In one of our main results, we shall also suppose that $M$ supports the Poincar\'e inequality, namely that:
\begin{equation}\label{P}
(\textrm{Poincar\'e\ inequality)}\ \ \ \ \ \|v\|_{L^2(M)} \le \frac{1}{C_p} \|\nabla v\|_{L^2(M)} \quad \text{for any}\,\,\, v\in C_c^{\infty}(M),
\end{equation}
for some $C_p>0$. Observe that, for instance, \eqref{S} holds if $M$ is a Cartan-Hadamard manifold, i.e. a simply connected Riemannian manifold with nonpositive sectional curvatures, while \eqref{P} is valid when $M$ is a Cartan-Hadamard manifold satisfying the additional condition of having sectional curvatures bounded above by a constant $-c<0$ (see, e.g., \cite{Grig, Grig3}). Therefore, as is well known, in $\mathbb R^N$ \eqref{S} holds, but \eqref{P} fails, whereas on the hyperbolic space both \eqref{S} and \eqref{P} are fulfilled.

\smallskip

\subsection{On some existing results} In \cite{GMeP} problem \eqref{problema} has been studied when $p<m$. We refer the reader to such paper for a comprehensive account of the literature; here we limit ourselves to recall some results particularly related to ours.

For $M=\mathbb R^N$ and $m=1$, it is well-known that, if $p\leq 1+\frac 2 N$, then the solution of problem \eqref{problema} blows up in finite time for any $u_0\not\equiv 0$, while global existence holds if $p>1+\frac 2 N$ and $u_0$ is bounded and small enough (see \cite{F, H}; for further results see also \cite{CFG, FI, GV, L, MQV, Q,  S, Vaz1, W, Y}). For $m>1$, in \cite{SGKM} it is shown that the solution to problem \eqref{problema} blows up for any $p\leq m+\frac 2 N,  u_0\not\equiv 0$; instead, there exists a global in time solution provided $p>m+\frac 2N$ and $u_0$ is compactly supported and sufficiently small. On Riemannian manifolds satisfying suitable volume growth conditions, for $m=1$ and $p\leq 1+\frac 2N$, in \cite{MMP, Z} it is proved that the solution of problem \eqref{problema} blows up for any $u_0\not\equiv 0$, while global existence holds if $p>1+\frac 2 N$ for small enough initial data $u_0$. Similar results have also been stablished in \cite{BPT, Pu3, WY, WY2}.

Problem \eqref{problema}, without the forcing term $u^p$, has been largely studied on Riemannian manifolds, and in particular on Cartan-Hadamard manifolds, in \cite{BG, GIM, GMhyp, GM2, GMPbd, GMPrm, GMV, Pu1, VazH}.
In \cite{GMPv} problem \eqref{problema} is addressed on Cartan-Hadamard manifolds with $-k_1\leq \operatorname{sec}\leq -k_2$ for some $k_1>k_2>0$, where $\operatorname{sec}$ denotes the sectional curvature. It is shown that, for any $p>m$, there exists a global in time solution, provided that $u_0$ has compact support and is small enough, while if $u_0$ is large enough, then there exists a solution blowing up in finite time.

For any $x_0\in M, r>0$ let $B_r(x_0)$ be the geodesic ball centered in $x_0$ and radius $r$, let $g_{ij}$ the metric tensor. In \cite{Z}, problem \eqref{problema} is studied when $M$ is a manifold with a pole, $\mu(B_r(x_0))\leq C r^{\alpha}$ for some $\alpha>2$ and $C>0$. Under an additional smallness condition on curvature at infinity, if
$u_0$ is sufficiently small and with compact support, then there exists a global solution to problem \eqref{problema}. Global existence is also proved, for some initial data $u_0$, under the assumption that $M$ has nonnegative Ricci curvature and $p>\frac{\alpha}{\alpha-2}m$. It should be noticed that such result do not cover cases in which negative curvature either does not tend to zero at infinity, or does so not sufficiently fast, in particular the case of the hyperbolic space cannot be addressed.

Finally, in \cite{GMeP} global existence of solutions to problem \eqref{problema} is obtained, for any $p<m$ and $u_0\in L^m(M)$, under the assumption that the Sobolev and the Poincaré inequalities hold on $M$.

\medskip

\subsection{Qualitative statements of our new results in the Riemannian setting} Our results concerning problem \eqref{problema} can be summarized as follows.
\begin{itemize}
\item (See Theorem \ref{teo22}) We prove global existence of solutions to \eqref{problema}, assuming that the initial datum is sufficiently small, that
\[p> m + \frac 2 N,\]
and that the Sobolev inequality \eqref{S} holds; moreover, smoothing effects and the fact that suitable $L^q$ norms of solutions decrease in time are obtained. To be specific, any sufficiently small initial datum $u_0\in L^m(M)\cap L^{(p-m)\frac N2}(M)$ gives rise to a global solution $u(t)$ such that $u(t)\in L^{\infty}(M)$ for all $t>0$ with a quantitative bound on the $L^{\infty}$ norm of the solution.
\item (See Theorem \ref{teo71}) We show that, if both the Sobolev and the Poincaré inequality (i.e. \eqref{S}, \eqref{P}) hold, then for any
\[p>m,\]
for any sufficiently small initial datum $u_0$, belonging to suitable Lebesgue spaces, there exists a global solution $u(t)$ such that $u(t)\in L^{\infty}(M)$. Furthermore,  a quantitative bound for the $L^{\infty}$ norm of the solution is satisfied for all $t>0$.
\end{itemize}
Note that in Theorem \ref{teo22} we only assume the Sobolev inequality and we require that $p>m+\frac 2 N$, instead in Theorem \ref{teo71} we can relax the assumption on the exponent $p$, indeed we assume $p>m$, but we need to further require that the Poincaré inequality holds. Moreover, in the two theorems, the hypotheses on the initial data are different.

\bigskip The main results given in Theorems \ref{teo22} and \ref{teo71} depend essentially only on the validity of inequalities \eqref{S} and \eqref{P}, are functional analytic in character and hence can be generalized to different contexts.

\subsection{The case of Euclidean, weighted diffusion}
As a particularly significant setting, we single out the case of Euclidean, mass-weighted reaction diffusion equations, that has been the object of intense research. In fact we consider the problem
\begin{equation}\label{problema2}
\begin{cases}
\rho\, u_t= \Delta u^m +\rho\, u^p & \text{in}\,\, \R^N\times (0,T) \\
u\,\,  =u_0 &\text{in}\,\, \R^N\times \{0\},
\end{cases}
\end{equation}
where $\rho:\R^N\to\R$ is strictly positive, continuous and  bounded, and represents a \it mass variable density \rm. The problem is naturally posed in the weighted spaces
$$L^q_{\rho}(\R^N)=\left\{v:\R^N\to\R\,\, \text{measurable}\,\,  ,   \,\, \|v\|_{L^q_{\rho}}:=\left(\int_{\R^N} \,v^q\rho(x)\,dx\right)^{1/q}<+\infty\right\}.$$

This kind of problem arises in a physical model provided in \cite{KR}. Such choice of $\rho$ ensures that the following analogue of \eqref{S} holds:
\begin{equation}\label{S-pesi}
\|v\|_{L^{2^*}_{\rho}(\R^N)} \le \frac{1}{C_s} \|\nabla v\|_{L^2(\R^N)}\quad \text{for any}\,\,\, v\in C_c^{\infty}(\R^N)
\end{equation}
for a suitable positive constant $C_s$. In some cases we also assume that the weighted Poincaré inequality is valid, that is
\begin{equation}\label{P-pesi}
\|v\|_{L^2_{\rho}(\mathbb R^N)} \le \frac{1}{C_p} \|\nabla v\|_{L^2(\mathbb R^N)} \quad \text{for any}\,\,\, v\in C_c^{\infty}(\mathbb R^N),
\end{equation}
for some $C_p>0$. For example, \eqref{P-pesi} is fulfilled when $\rho(x)\asymp  |x|^{-a}$, as $|x|\to +\infty$, for every $a\geq 2$, whereas, \eqref{S-pesi} is valid for every $a>0$.

\smallskip

Problem \eqref{problema2} under the assumption $1<p<m$ has been investigated in \cite{GMeP}. Under the assumption that the Poincar\'{e} inequality is valid on $M$,  it is shown that global existence and a smoothing effect for small $L^m$ initial data hold, that is solutions corresponding to such data are bounded for all positive times with a quantitative bound on their $L^\infty$ norm.

\smallskip

In \cite{MT, MTS} problem \eqref{problema2} is also investigated, under certain conditions on $\rho$. It is proved that if $\rho(x)=|x|^{-a}$ with $a\in (0,2)$,
\[p>m+\frac{2-a}{N-a},\]
and $u_0\geq 0$ is small enough, then a global solution exists (see \cite[Theorem 1]{MT}). Note that the homogeneity of the weight $\rho(x)=|x|^{-a}$ is essentially used in the proof, since the Caffarelli-Kohn-Nirenberg estimate is exploited, which requires such a type of weight. In addition, a smoothing estimate holds. On the other hand, any nonnegative solution blows up, in a suitable sense, when $\rho(x)=|x|^{-a}$
 or $\rho(x)=(1+|x|)^{-a}$ with $a\in [0,2)$, $u_0\not\equiv 0$ and
 \[1<p<m+\frac{2-a}{N-a}.\]
Furthermore, in \cite{MTS, MTS2}, such results have been extended to more general initial data, decaying at infinity with a certain rate (see \cite{MTS}).
Finally, in \cite[Theorem 2]{MT}, it is shown that if $p>m$, $\rho(x)=(1+|x|)^{-a}$ with $a>2$, and $u_0$ is small enough, a global solution exists.

Problem \eqref{problema2} has also been studied in \cite{MP1}, \cite{MP2}, by means of suitable barriers, supposing that the initial datum is continuous and with compact support. In particular, in \cite{MP1} the case that $\rho(x)\asymp |x|^{-a}$ for $|x|\to+\infty$
with $a\in (0,2)$ is addressed. It is proved that for any $p>1$, if $u_0$ is large enough, then the solution blows up in finite time. On the other hand, if $p>\bar p$, for a certain
$\bar p>m$ depending on $m, p$ and $\rho$, and $u_0$ is small enough, then there exists a global bounded solution. Moreover, in \cite{MP2} the case that $a\geq 2$ is investigated. For $a=2$, blowup is shown to occur when $u_0$ is big enough, whereas global existence holds when $u_0$ is small enough. For $a>2$ it is proved that if $p>m$, $u_0\in L^{\infty}_{\rm{loc}}(\mathbb R^N)$ and goes to $0$ at infinity with a suitable rate, then there exists a global bounded solution. Furthermore, for the same initial datum $u_0$, if $1<p<m$, then there exists a global solution, which could blow up as $t\to +\infty$\,.

\smallskip

Our main results concerning problem \eqref{problema2} can be summarized as follows. Assume that $\rho\in C(\mathbb R^N)\cap L^{\infty}(\mathbb R^N), \rho>0$.
\begin{itemize}
\item (See Theorem \ref{teo24}) We prove that  \eqref{problema2} admits a global solution, provided that
\[p> m + \frac 2N;\]
moreover, certain smoothing effects for solutions are fulfilled. More precisely, for any sufficiently small initial datum $u_0\in L^m_{\rho}(\mathbb R^N)\cap L^{(p-m)\frac N2}_{\rho}(\mathbb R^N)$ there exists a global solution $u(t)$ such that $u(t)\in L^{\infty}(\mathbb R^N)$ for all $t>0$ and a quantitative bound on the $L^{\infty}$ norm is verified. Moreover, suitable $L^q$ norms of solutions decrease in time.
\item (See Theorem \ref{teo71W}) We show that, if the Poincaré inequality \eqref{P-pesi} holds and one assumes the condition
\[p>m,\]
then, for any sufficiently small initial datum $u_0$ belonging to suitable Lebesgue spaces, there exists a global solution $u(t)$ to \eqref{problema2} such that $u(t)\in L^{\infty}(\mathbb R^N)$, with a quantitative bound on the $L^{\infty}$ norm.
\end{itemize}

Let us compare our results with those in \cite{MT}. Theorem \ref{teo24} deals with a different class of weights $\rho$ with respect to \cite[Theorem 1]{MT}, where $\rho(x)=|x|^{-a}$ for $a\in (0,2)$, and the homogeneity of $\rho$ is used. As a consequence, also the hypotheses on $p$ and the methods of proofs are different.  Furthermore, Theorem \ref{teo71W} requires the validity of the Poincaré inequality, hence, in particular, it can be applied when $\rho(x)=(1+|x|)^{-a}$ with $a\geq 2$ (see \cite{GMPo}). On the other hand, in Theorem \cite[Theorem 2]{MT} it is assumed that $\rho(x)=(1+|x|)^{-a}$ for $a>2$, so, the case $a=2$ is not included.

\medskip

\subsection{Organization of the paper} In Section \ref{statements} we state all our main results. In Section \ref{elliptic} some auxiliary results concerning elliptic problems are deduced together with a Benilan-Crandal type estimate. In Section \ref{Lp} we introduce a family of approximating problems. Then, for such solutions, we prove that suitable $L^q$ norms of solutions decrease in time, and a smoothing estimate, in the case $p>m+\frac 2 N$, supposing that $M$ supports the Sobolev inequality. Under such assumptions, global existence for problem \eqref{problema} is shown in Section \ref{proofs}. In Section \ref{Lpbis} we prove that suitable $L^q$ norms of solutions decrease in time, and $L^\infty$ bounds for solutions of the approximating problems, under the assumptions that $p>m$ and that $M$ supports the Poincaré inequality as well. Then, under such hypotheses, existence of global solutions to problem \eqref{problema} is proved.  Finally, a concise proof of the results concerning problem \eqref{problema2} is given in Section \ref{weight} by adapting the previous methods to that situation.

\section{Statements of main results}\label{statements}
We state first our results concerning solutions to problem \eqref{problema}, then we pass to the ones valid for solutions to problem \eqref{problema2}.

\subsection{Global existence on Riemannian manifolds}
Solutions to \eqref{problema} will be meant in the very weak, or distributional, sense, according to the following definition.

\begin{definition}\label{def21}
Let $M$ be a complete noncompact Riemannian manifold of infinite volume. Let $m>1$, $p>m$ and $u_0\in{\textrm L}^{1}_{\textit{loc}}(M)$, $u_0\ge0$. We say that the function $u$ is a solution to problem \eqref{problema} in the time interval $[0,T)$ if
$$
u\in L^p_{loc}(M\times(0,T)) 
$$
and for any $\varphi \in C_c^{\infty}(M\times[0,T])$ such that $\varphi(x,T)=0$ for any $x\in M$, $u$ satisfies the equality:
\begin{equation*}
\begin{aligned}
-\int_0^T\int_{M} \,u\,\varphi_t\,d\mu\,dt =&\int_0^T\int_{M} u^m\,\Delta\varphi\,d\mu\,dt\,+ \int_0^T\int_{M} \,u^p\,\varphi\,d\mu\,dt \\
& +\int_{M} \,u_0(x)\,\varphi(x,0)\,d\mu.
\end{aligned}
\end{equation*}
\end{definition}

First we consider the case that $p>m+\frac 2 N$ and the Sobolev inequality holds on $M$. In order to state our results we define
\begin{equation}\label{p0}p_0:=(p-m)\frac{N}{2}.\end{equation} Observe that $p_0>1$ whenever $p>m+\frac 2N$.


\begin{theorem}\label{teo22}
Let $M$ be a complete, noncompact manifold of infinite volume such that the Sobolev inequality \eqref{S} holds. Let $m>1$, $p>m+\frac{2}{N}$ and $u_0\in{\textrm L}^m(M)\cap{\textrm L}^{p_0}(M)$, $u_0\ge0$ where $p_0$ has been defined in \eqref{p0}. Let
$$
r>\,\max\left\{p_0,\, \frac N2\right\},\quad\quad s=1+\frac 2N-\frac 1r.
$$
Assume that
\begin{equation}\label{epsilon0}
\|u_0\|_{\textrm L^{p_0}(M)}\,<\,\varepsilon_0
\end{equation}
with $\varepsilon_0=\varepsilon_0(p,m,N,r, C_s)$ sufficiently small. Then problem \eqref{problema} admits a solution for any $T>0$,  in the sense of Definition \ref{def21}. Moreover, for any $\tau>0,$ one has $u\in L^{\infty}(M\times(\tau,+\infty))$ and there exists a numerical constant $\Gamma>0$ such that, for all $t>0$, one has
\begin{equation*}
\|u(t)\|_{L^{\infty}(M)}
\le \Gamma\, t^{-\frac{\gamma}{ms}}\left\{\|u_0\|_{L^{p_0}(M)}^{\delta_{1}}+\|u_0\|_{L^{p_0}(M)}^{\delta_{2}} \right\}^{\frac{1}{ms}}\|u_0\|_{L^{m}(M)}^{\frac{s-1}{s}},
\end{equation*}
where
\begin{equation*}
\gamma= \frac{p}{p-1}\left[1-\frac{N(p-m)}{2\,p\,r}\right],\quad\delta_{1}=p\,\frac{p-m}{p-1}\left[1+\frac{N(m-1)}{2\,p\,r}\right],\quad \delta_{2}=\frac{p-m}{p-1}\left[1+\frac{N(m-1)}{2\,r}\right].
\end{equation*}
Moreover, let  $p_0\le q<\infty$ 
and
\begin{equation}\label{eps3a}
\|u_0\|_{L^{p_0}(M)}< \hat \varepsilon_0
\end{equation}
for $\hat\varepsilon_0=\hat\varepsilon_0(p, m , N, r, C_s, q)$ small enough. Then there exists a constant $C=C(m,p,N,\varepsilon_0,C_s, q)>0$ such that
\begin{equation}\label{eq23}
\|u(t)\|_{L^q(M)}\le C\,t^{-\gamma_q} \|u_{0}\|^{\delta_q}_{L^{p_0}(M)}\quad \textrm{for all }\,\, t>0\,,
\end{equation}
where
$$
\gamma_q=\frac{1}{p-1}\left[1-\frac{N(p-m)}{2q}\right],\quad \delta_q=\frac{p-m}{p-1}\left[1+\frac{N(m-1)}{2q}\right]\,.
$$
Finally, for any $1<q<\infty$, if $u_0\in {\textrm L}^q(M)\cap\textrm L^{p_0}(M)\cap L^m(M)$ and
\begin{equation}\label{epsilon1}
\|u_0\|_{\textrm L^{p_0}(M)}\,<\,\varepsilon 
\end{equation}
with $\varepsilon=\varepsilon(p,m,N,r, C_s,q)$ sufficiently small, then
\begin{equation}\label{eq22}
\|u(t)\|_{L^q(M)}\le  \|u_{0}\|_{L^q(M)}\quad \textrm{for all }\,\, t>0\,.
\end{equation}

\end{theorem}

\begin{remark}\label{remark4}\rm
We notice that the proof of the above theorem will show that one can take an explicit value of $\varepsilon_0$ in \eqref{epsilon0}. In fact,
let $q_0>1$ be fixed and $\{q_n\}_{n\in\mathbb{N}}$ be the sequence defined by:
\begin{equation*}
\begin{aligned}
q_n=\frac{N}{N-2}(m+q_{n-1}-1), \ \ \ \ \forall n\in\mathbb{N},
\end{aligned}
\end{equation*}
so that
\begin{equation}\label{eq400}
q_n=\left(\frac{N}{N-2}\right)^{n}q_0+\frac{N(m-1)}{N-2} \sum_{i=0}^{n-1} \left(\frac{N}{N-2}\right)^i.
\end{equation}
Clearly, $\{q_n\}$ is increasing and $q_n \longrightarrow +\infty$ as $n\to +\infty$.
Fix $q\in [q_0,+\infty)$ and let $\bar n$ be the first index such that $q_{\bar n}\ge q$. Define
\begin{equation}\label{eq20}
\tilde \varepsilon_0=\tilde \varepsilon_0(p,m,N,C_s,q, q_0):=\left[\min \left\{\min_{n=0,...,\bar n}\frac{2m( q_n-1)}{(m+q_n-1)^2}C_s^2;\,\,\frac{2m( p_0-1)}{(m+p_0-1)^2}C_s^2\right\}\right]^{\frac1{p-m}}.
\end{equation}
Observe that $\varepsilon_0$ in \eqref{eq20} depends on the value of $q$ through the sequence $\{q_n\}$. More precisely, $\bar n$ is increasing with respect to $q$, while the quantity $\min_{n=0,...,\bar n}\frac{2m( q_n-1)}{(m+q_n-1)^2}C_s^2$ decreases w.r.t. $q$. 
We then let $q_0=p_0$, take $q=pr$ and define, for these choice of $q_0,q$,
\[\varepsilon_0=\varepsilon_0(p, m, N, C_s, r)=\tilde \varepsilon_0(p, m, N, C_s, pr, p_0)\,.\]

Furthermore, in \eqref{eps3a} we can take
\begin{equation}\label{epsilon0aaa}
\hat \varepsilon_0=\hat\varepsilon_0(p, m , N, C_s, q)=\tilde\varepsilon_0(p, m, N, C_s, q, p_0)\,.
\end{equation}

Similarly, one can choose the following explicit value for $\varepsilon$ in \eqref{epsilon1}:
\begin{equation}\label{epsilon000}
\varepsilon= \bar\varepsilon\wedge \varepsilon_0,
\end{equation}
where
\[\bar\varepsilon=\bar\varepsilon(p, m, C_s, q):=\left[\min \left\{\frac{2m(q-1)}{(m+q-1)^2}C_s^2;\,\,\frac{2m(p_0-1)}{\left(m+p_0-1\right)^2}C_s^2\right\}\right]^{\frac1{p-m}}\,.\]
\end{remark}

\begin{remark}\rm
Observe that, for $M=\mathbb R^N$, in \cite[Theorem 3, pag. 220]{SGKM}  it is shown that if $p>m+\frac 2 N$ and $u_0$ has compact support and is small enough, then the solution to problem \eqref{problema} globally exists and decays like $$t^{-\frac{1}{p-1}} \quad \textrm{ as}  \quad t\longrightarrow +\infty.$$ Note that under these assumptions, Theorem \ref{teo22} can be applied. It implies that the solution to problem \eqref{problema} globally exists and decays like
$$t^{-\frac{\gamma}{ms}} \quad \textrm{  as }  \quad t\longrightarrow +\infty.$$
It is easily seen that, for any $p\ge m\left(1+\frac{2}{N}\right)$,
$$\frac{\gamma}{ms}\,\ge\, \frac{1}{p-1};$$
instead, for any $m+\frac 2 N<p<m\left(1+\frac{2}{N}\right)$,
\[\frac{\gamma}{ms}\,<\, \frac{1}{p-1}.\]
Hence, when $p\ge m\left(1+\frac{2}{N}\right)$ the decay's rate of the solution $u(t)$, for large times, given by Theorem \ref{teo22} is better than that of  \cite[Theorem 3, pag. 220]{SGKM}, while the opposite is true for $m+\frac 2 N<p<m\left(1+\frac{2}{N}\right)$. In both cases, the class of initial data considered in Theorem \ref{teo22} is wider.
\end{remark}

\medskip

In the next theorem, we address the case that $p>m$, supposing that both the inequalities \eqref{S} and \eqref{P} hold on $M$.

\begin{theorem}\label{teo71}
Let $M$ be a complete, noncompact manifold of infinite volume such that the Sobolev inequality \eqref{S} and the Poincaré inequality \eqref{P} hold. Let
$$
m>1,\quad p>m, \quad r>\,\frac N2,
$$
and $u_0\in{\textrm L}^{\theta}(M)\cap{\textrm L}^{pr}(M)$ where $\theta=\min\{m,r\}$, $u_0\ge0$. Let
$$
s=1+\frac 2N-\frac 1r.
$$
Assume that
\begin{equation}\label{epsilon11}
\left\| u_0\right\|_{L^{p\frac N2}(M)}\,<\,\varepsilon_1
\end{equation}
holds with $\varepsilon_1=\varepsilon_1(m,p,N,r, C_p,C_s)$ sufficiently small. Then problem \eqref{problema} admits a solution for any $T>0$,  in the sense of Definition \ref{def21}. Moreover for any $\tau>0$ one has $u\in L^{\infty}(M\times(\tau,+\infty))$ and for all $t>0$ one has
\begin{equation*}
\|u(t)\|_{L^{\infty}(M)}
\le\left(\frac{s}{s-1}\right)^{\frac 1m}\|u_{0}\|_{L^{m}(M)}^{\frac{s-1}{s}}\left[ \|u_{0}\|_{L^{pr}(M)}^{p}+\frac{1}{(m-1)t}\|u_{0}\|_{L^{r}(M)}\right]^{\frac{1}{ms}}.
\end{equation*}
Moreover, suppose that $u_0\in {\textrm L}^q(M)\cap L^{\theta}(M)\cap L^{pr}(M)$ for some for $1<q<\infty$,
\begin{equation}\label{epsilon2a}
\|u_0\|_{L^{p\frac N2}(M)}<\varepsilon_2,
\end{equation}
for some $\varepsilon_2=\varepsilon_2(p, m ,N, r, C_p, C_s, q)$ sufficiently small. Then
\begin{equation}\label{eq721}
\|u(t)\|_{L^q(M)}\le  \|u_{0}\|_{L^q(M)}\quad \textrm{for all }\,\, t>0\,.
\end{equation}

\end{theorem}

\begin{remark}\label{remark5}\rm
We define, given $q>1$:
\begin{equation}\label{hp}
\tilde\varepsilon_1(q):=\left[\min \left\{\frac{2m(q-1)}{(m+q-1)^2}C;\,\frac{2m\left(p\frac N2-1\right)}{\left(m+p\frac N2-1\right)^2}C\right\}\right]^\frac{p+m+q-1}{p(p+q-1)-m(m+q-1)}
\end{equation}
where $C=C_p^{2m/p}\,\tilde C$ and  $\tilde C=\tilde C(C_s,m,p,q)>0$ is defined in \eqref{tildec} below, with the choice $\theta:=\frac{m(m+q-1)}{p(p+q-1)}$. The proof will show that one can choose $\varepsilon_1:=\min_{i=1,\ldots,4}\tilde\varepsilon_1(q_i)$ where $q_1=m$, $q_2=p$, $q_3=pr$ and $q_4=r$.

Similarly, we observe that in \eqref{epsilon2a} we can choose
\begin{equation}\label{hp2a}
\varepsilon_2=\varepsilon_1 \wedge \tilde \varepsilon_1(q)\,.
\end{equation}
\end{remark}

In the next sections we always keep the notation as in Remarks \ref{remark4} and \ref{remark5}.

\subsection{Weighted, Euclidean reaction-diffusion problems}
We consider a \it weight \rm $\rho:\R^N\to\R$ such that
\begin{equation}\label{rho2}
\rho \in C(\R^N)\cap L^{\infty}(\R^N), \ \ \rho(x)>0 \,\, \text{for any}\,\, x\in \R^N.
\end{equation}

Solutions to problem \eqref{problema2} are meant according to the following definition.
\begin{definition}\label{def23}
Let $m>1$, $p>m$ and $u_0\in{\textrm L}^{1}_{\rho,\textit{loc}}(\R^N)$, $u_0\ge0$. Let the weight $\rho$ satisfy \eqref{rho2}. We say that the function $u$ is a solution to problem \eqref{problema2} in the interval $[0, T)$ if
$$
u\in L^p_{\rho,loc}(\mathbb R^N\times(0,T))\,\,\,
$$
and for any $\varphi \in C_c^{\infty}(\mathbb R^N\times[0,T])$ such that $\varphi(x,T)=0$ for any $x\in \mathbb R^N$, $u$ satisfies the equality:
\begin{equation*}
\begin{aligned}
-\int_0^T\int_{\mathbb{R}^N} \,u\,\varphi_t\,\rho(x)\,dx\,dt =&\int_0^T\int_{\mathbb R^N} u^m\,\Delta \varphi\,dx\,dt\,+ \int_0^T\int_{\mathbb R^N} \,u^p\,\varphi\,\rho(x)\,dx\,dt \\
& +\int_{\mathbb R^N} \,u_0(x)\,\varphi(x,0)\,\rho(x)\,dx.
\end{aligned}
\end{equation*}
\end{definition}

First we consider the case that $p>m+\frac 2 N$. Recall that since $\rho$ is bounded, the Sobolev inequality \eqref{S-pesi} necessarily holds.


\begin{theorem}\label{teo24}

Let $\rho$ satisfy \eqref{rho2}. Let $m>1$, $p>m+\frac{2}{N}$ and $u_0\in L^{m}_{\rho}(\R^N)\cap{\textrm L}_{\rho}^{p_0}(\mathbb R^N)$, $u_0\ge0$ with $p_0$ defined in \eqref{p0}. Let
$$
r>\,\max\left\{p_0, \frac N2\right\}, \quad\quad s=1+\frac{2}{N}-\frac 1r.
$$
Assume that
\begin{equation*}
\|u_0\|_{{\textrm L}_{\rho}^{p_0}(\mathbb R^N)}\,<\,\varepsilon_0
\end{equation*}
holds, with $\varepsilon_0=\varepsilon_0(p,m,N,r,C_s)$ sufficiently small. Then problem \eqref{problema2} admits a solution for any $T>0$,  in the sense of Definition \ref{def23}. Moreover, for any $\tau>0,$ one has $u\in L^{\infty}(\mathbb R^N\times(\tau,+\infty))$ and there exist $\Gamma>0$ such that, for all $t>0$, one has
\begin{equation*}
\|u(t)\|_{L^{\infty}(\R^N)}
\le \Gamma\, t^{-\frac{\gamma}{ms}}\left\{\|u_0\|_{L^{p_0}_{\rho}(\R^N)}^{\delta_{1}}+\frac{1}{m-1}\,\|u_0\|_{L^{p_0}_{\rho}(\R^N)}^{\delta_{2}} \right\}^{\frac{1}{ms}}\|u_0\|_{L^{m}_{\rho}(\R^N)}^{\frac{s-1}{s}},
\end{equation*}
where
\begin{equation*}
\gamma= \frac{p}{p-1}\left[1-\frac{N(p-m)}{2\,p\,r}\right],\,\,\delta_1=p\,\frac{p-m}{p-1}\left[1+\frac{N(m-1)}{2\,p\,r}\right],\,\delta_2=\frac{p-m}{p-1}\left[1+\frac{N(m-1)}{2\,r}\right]
\end{equation*}
Moreover, let  $p_0\le q<\infty$ 
and
\begin{equation*}
\|u_0\|_{L^{p_0}_{\rho}(\mathbb R^N)}< \hat \varepsilon_0
\end{equation*}
for $\hat\varepsilon_0=\hat\varepsilon_0(p, m , N, r, C_s, q)$ small enough.
Then there exists a constant $C=C(m,p,N,\varepsilon_0,C_s, q)>0$ such that
\begin{equation*}
\|u(t)\|_{L_{\rho}^q(\R^N)}\le C\,t^{-\gamma_q} \|u_{0}\|^{\delta_q}_{L_{\rho}^{p_0}(\R^N)}\quad \textrm{for all }\,\, t>0\,,
\end{equation*}
where
$$
\gamma_q=\frac{1}{p-1}\left[1-\frac{N(p-m)}{2q}\right],\quad \delta_q=\frac{p-m}{p-1}\left[1+\frac{N(m-1)}{2q}\right]\,.
$$

Finally, for any $1<q<\infty$, if $u_0\in  \textrm L_{\rho}^q(\R^N)\cap\textrm L_{\rho}^{p_0}(\R^N)\cap \textrm L_{\rho}^{m}(\R^N)$ and
\begin{equation*}
\|u_0\|_{\textrm L^{p_0}_{\rho}(\mathbb R^N)}\,<\,\varepsilon 
\end{equation*}
holds, with $\varepsilon=\varepsilon(p,m,N, r, C_s,q)$ sufficiently small, then
\begin{equation*}
\|u(t)\|_{L_{\rho}^q(\R^N)}\le  \|u_{0}\|_{L_{\rho}^q(\R^N)}\quad \textrm{for all }\,\, t>0\,.
\end{equation*}

\end{theorem}

A quantitative form of the smallness condition on $u_0$ in the above theorem can be given exactly as in Remark \ref{remark4}, see in particular \eqref{eq20}, \eqref{epsilon0aaa} and \eqref{epsilon000}.

In the next theorem, we address the case $p>m$. We suppose that the Poincaré inequality \eqref{P-pesi} holds.

\begin{theorem}\label{teo71W}
Let $\rho$ satisfy \eqref{rho2} and assume that the inequality 
\eqref{P-pesi} hold. Let
$$
m>1,\quad p>m, \quad r>\,\frac N2,
$$
and $u_0\in{\textrm L}^{\theta}_{\rho}(\R^N)\cap{\textrm L}^{pr}_{\rho}(\R^N)$ where $\theta=\min\{m,r\}$, $u_0\ge0$. Let
$$
s=1+\frac 2N-\frac 1r.
$$
Assume that
\[
\left\| u_0\right\|_{L_\rho^{p\frac N2}({\mathbb R}^N)}\,<\,\varepsilon_1
\]
holds with $\varepsilon_1=\varepsilon_1(m,p,N,r, C_p,C_s)$ sufficiently small
Then problem \eqref{problema2} admits a solution for any $T>0$,  in the sense of Definition \ref{def23}. Moreover, for any $\tau>0$ one has $u\in L^{\infty}(\R^N\times(\tau,+\infty))$ and for all $t>0$ one has
\begin{equation*}
\|u(t)\|_{L^{\infty}(\R^N)}
\le\left(\frac{s}{s-1}\right)^{\frac 1m}\|u_{0}\|_{L^{m}_{\rho}(\R^N)}^{\frac{s-1}{s}}\left[ \|u_{0}\|_{L^{pr}_{\rho}(\R^N)}^{p}+\frac{1}{(m-1)t}\|u_{0}\|_{L^{r}_{\rho}(\R^N)}\right]^{\frac{1}{ms}}.
\end{equation*}
Moreover, suppose that $u_0\in {\textrm L}^q_{\rho}(\mathbb R^N)\cap  {\textrm L}^\theta_{\rho}(\mathbb R^N)\cap  {\textrm L}^{pr}_{\rho}(\mathbb R^N)$ for some for $1<q<\infty$,
\[\|u_0\|_{L^{p\frac N2}_{\rho}(\mathbb R^N) }< \varepsilon_2, \]
for some $\varepsilon_2=\varepsilon_2(p, m, N, r, C_p, C_s, q)$ small enough. Then
\begin{equation*}
\|u(t)\|_{L^q_{\rho}(\R^N)}\le  \|u_{0}\|_{L^q_{\rho}(\R^N)}\quad \textrm{for all }\,\, t>0\,.
\end{equation*}
\end{theorem}

A quantitative form of the smallness condition on $u_0$ in the above Theorem can be given exactly as in Remark \ref{remark5}, see in particular \eqref{hp} and \eqref{hp2a}.


\section{Auxiliary results for elliptic problems}\label{elliptic}

Let $x_0,x \in M$. We denote by $r(x)=\textrm{dist}\,(x_0,x)$ the Riemannian distance between $x_0$ and $x$. Moreover, we let $B_R(x_0):=\{x\in M, \textrm{dist}\,(x_0,x)<R\}$ be the geodesics ball with centre $x_0 \in M$ and radius $R > 0$. If a reference point $x_0\in M$ is fixed, we shall simply denote by $B_R$ the ball with centre $x_0$ and radius $R$. Moreover we denote by $\mu$ the Riemannian measure on $M$.

\smallskip

For any given function $v$, we define for any $k\in\R^+$
\begin{equation}\label{31}
T_k(v):=\begin{cases} &k\quad \text{if}\,\,\, v\ge k \\ &v \quad \text{if}\,\,\, |v|< k \\ &-k\quad \text{if}\,\,\, v\le -k\end{cases}\,\,.
\end{equation}

For every $R>0$, $k>0,$ consider the problem
\begin{equation}\label{eq31}
\begin{cases}
\, u_t= \Delta u^m +\, T_k(u^p) & \text{in}\,\, B_R\times (0,+\infty) \\
u=0 &\text{in}\,\, \partial B_R\times (0,+\infty)\\
u=u_0 &\text{in}\,\, B_R\times \{0\}, \\
\end{cases}
\end{equation}
where $u_0\in L^\infty(B_R)$, $u_0\geq 0$.
Solutions to problem \eqref{eq31} are meant in the weak sense as follows.

\begin{definition}\label{def31}
Let $m>1$ and $p>m$. Let $u_0\in L^\infty(B_R)$, $u_0\geq 0$. We say that a nonnegative function $u$ is a solution to problem \eqref{eq31} if
$$
u\in L^{\infty}(B_R\times(0,+\infty)), \,\,\, u^m\in L^2\big((0, T); H^1_0(B_R)\big) \quad\quad \textrm{ for any }\, T>0,
$$
and for any $T>0$, $\varphi \in C_c^{\infty}(B_R\times[0,T])$ such that $\varphi(x,T)=0$ for every $x\in B_R$, $u$ satisfies the equality:
\begin{equation*}
\begin{aligned}
-\int_0^T\int_{B_R} \,u\,\varphi_t\,d\mu\,dt =&- \int_0^T\int_{B_R} \langle \nabla u^m, \nabla \varphi \rangle \,d\mu\,dt\,+ \int_0^T\int_{B_R} \,T_k(u^p)\,\varphi\,d\mu\,dt \\
& +\int_{B_R} \,u_0(x)\,\varphi(x,0)\,d\mu.
\end{aligned}
\end{equation*}
\end{definition}
We also consider elliptic problems of the type
\begin{equation}\label{eq32}
\begin{cases}
-\Delta u &= f \quad \textrm{ in }\,\, B_R\\
\;\quad u & = 0 \quad \textrm{ in }\,\, \partial B_R\,,
\end{cases}
\end{equation}
where $f\in L^q(B_R)$ for some $q>1$.

\begin{definition}\label{def32}
We say that $u\in H^1_0(B_R)$, $u\geq 0$ is a weak subsolution to problem \eqref{eq32} if
\[\int_{B_R}\langle \nabla u, \nabla \varphi \rangle\, d\mu \leq \int_{B_R} f\varphi\, d\mu,\]
for any $\varphi\in H^1_0(B_R), \varphi\geq 0$\,.
\end{definition}

In the next lemma we recall \cite[Lemma 3.6]{GMeP}, which will be used later.

\begin{lemma}\label{lemma1}
Let $v\in L^1(B_R)$. Let $\overline k>0$. Suppose that there exist $C>0$ and $s>1$ such that
\begin{equation*}
g(k)\le C\mu(A_k)^{s} \quad \text{for any}\,\,k\ge \bar k.
\end{equation*}
Then $v\in L^{\infty}(B_R)$ and
\begin{equation*}
\|v\|_{L^{\infty}(B_R)}\le\frac{s}{s-1} C^{\frac{1}{s}}\|v\|_{L^{1}(B_R)}^{1-\frac{1}{s}}+\bar k.
\end{equation*}
\end{lemma}

The following proposition contains an estimate in the spirit of the $L^\infty$ one of Stampacchia (see, e.g., \cite{KS}, \cite{BC} and references therein) in the ball $B_R$; however, some differences are in order. In fact, we aim at obtaining an estimate independent of the radius $R$ (see Remark \ref{remark2}). Since the volume of $M$ is infinite, the classical estimate of Stampacchia cannot be directly applied.

\begin{proposition}\label{prop1}
Let $f\in L^{m}(B_R)$ where $m>\frac N 2$.  Assume that $v\in H_0^1(B_R)$, $v\ge 0$ is a subsolution to problem
\begin{equation}\label{a1}
\begin{cases}
-\Delta v = f & \text{in}\,\, B_R\\
v=0 &\text{on}\,\, \partial B_R
\end{cases}.
\end{equation}
in the sense of Definition \ref{def32}. Then
\begin{equation}\label{eqa3}
\|v\|_{L^{\infty}(B_R)}\le\frac{s}{s-1}\left(\frac{1}{C_s}\right)^{\frac 2s}\|f\|_{L^m(B_R)}^{\frac 1s} \|v\|_{L^{1}(B_R)}^{\frac{s-1}{s}},
\end{equation}
\newline where
\begin{equation}\label{eq35}
s=1+\frac{2}{N}-\frac{1}{m}\,,
\end{equation}
\end{proposition}

\begin{remark}\label{remark2}
\rm If in Proposition \ref{prop1} we further assume that there exists a constant $k_0>0$ such that
\[\max\left\{\|v \|_{L^1(B_R)}, \|f \|_{L^{m}(B_R)}\right\} \leq k_0
\quad \textrm{ for all }\,\, R>0, \]
then from \eqref{eqa3}, we infer that the bound from above on $\|v\|_{L^{\infty}(B_R)}$ is independent of $R$. This fact will
have a key role in the proof of global existence for problem \eqref{problema}.
\end{remark}

\begin{proof}[Proof of Proposition \ref{prop1}]
We define
$$G_k(v):=v-T_k(v)$$ where $T_k(v)$ has been defined in \eqref{31} and $$A_k:=\{x\in B_R\,:\,|v(x)|> k\}.$$ Since $G_k(v)\in H^1_0(B_R)$ and $G_k(v)\geq 0$, we can take $G_k(v)$ as test function in problem \eqref{a1}. Arguing as in the proof of \cite[Proposition 3.3]{GMeP} we obtain
\begin{equation}\label{36}
\int_{B_R}|G_k(v)|\,d\mu\le\frac{1}{C_s^2}\|f\|_{L^{m}(B_R)}\mu(A_k)^{\frac{N+2}{N}-\frac 1m}.
\end{equation}
By \eqref{eq35}, setting
$$
C=\frac{1}{C_s^2}\|f\|_{L^m(B_R)},
$$
we rewrite \eqref{36} as
$$
\int_{B_R}|G_k(v)|\,d\mu\le\, C\mu(A_k)^s.
$$
Hence we can apply Lemma \ref{lemma1} to $v$ and we obtain
$$
\|v\|_{L^{\infty}(B_R)} \le C^{\frac 1s} \frac{s}{s-1} \|v\|_{L^1(B_R)}^{\frac{s-1}{s}}+\overline k.
$$
Taking the limit as $\overline k \longrightarrow 0$ and we get the thesis.

\end{proof}

We shall use the following Aronson-Benilan type estimate (see \cite{AB}; see also \cite[Proposition 2.3]{Sacks}).
\begin{proposition}\label{prop43}
Let $m>1$, $p>m$, $u_0\in H_0^1(B_R) \cap L^{\infty}(B_R)$, $u_0\ge 0$. Let $u$ be the solution to problem \eqref{eq31}. Then, for a.e. $t\in(0,T)$,
\begin{equation*}
-\Delta u^m(\cdot,t) \le u^p(\cdot, t)+\frac{1}{(m-1)t} u(\cdot,t) \quad \text{in}\,\,\,\mathfrak{D}'(B_R).
\end{equation*}
\end{proposition}
\begin{proof}
The conclusion follows by minor modifications of the proof of \cite[Proposition 2.3]{Sacks} (where $p<m)$, due to the fact that we have $p>m$. We define
$$
z=u_{t}+\frac{u}{m-1}\,
$$
and the operator
$$
Lz=\Delta \left(mu^{m-1}z\right)+mu^{p-1}z\,,
$$
where $u$ is the solution to problem \eqref{eq31}. Observe that
$$
\begin{aligned}
&z(x,0)\ge 0 \quad \text{for }\,\,\, x\in B_R\,,\\
&z(x,t)\ge 0 \quad \text{for }\,\,\, x\in \partial B_R\,\,\,\text{and }\,\,\,t\in (0, T)\,.
\end{aligned}
$$
Moreover, by direct computation, we get
$$
z_t-Lz\ge 0\quad \text{in}\,\,\, B_R\times(0, T).
$$
Thus, arguing as in \cite[Proposition 2.3]{Sacks}, thanks to the comparison principle, we get, for a.e. $t\in (0,T)$,
\[
-\Delta u^m(\cdot,t) \le T_k[u^p(\cdot, t)]+\frac{1}{(m-1)t} u(\cdot,t) \leq u^p(\cdot, t)+\frac{1}{(m-1)t} u(\cdot,t) \quad \text{in}\,\,\,\mathfrak{D}'(B_R),
\]
where we have used that $T_k(u^p)\leq u^p\,.$
\end{proof}

\section{$L^q$ and smoothing estimates for $p>m+\frac 2N$}\label{Lp}

\begin{lemma}\label{lemma41}
Let $m>1, p>m+\frac{2}{N}$.
Assume that inequality \eqref{S} holds. Suppose that $u_0\in L^{\infty}(B_R)$, $u_0\ge0$. Let $1<q<\infty$, $p_0$ as in \eqref{p0} and assume that
\begin{equation}\label{eq41}
\|u_0\|_{\textrm L^{p_0}(B_R)}\,<\,\bar\varepsilon
\end{equation}
with $\bar\varepsilon=\bar\varepsilon(p, m, q, C_s)$ sufficiently small.
Let $u$ be the solution of problem \eqref{eq31} in the sense of Definition \ref{def31}, such that in addition $u\in C([0,T), L^q(B_R))\, \text{for any} \ q\in(1,+\infty),\,\textrm{ for any }\, T>0$. Then
\begin{equation}\label{eq42}
\|u(t)\|_{L^q(B_R)} \le \|u_0\|_{L^q(B_R)}\quad \textrm{ for all }\,\, t>0\,.
\end{equation}
\end{lemma}

\begin{proof}
Since $u_0$ is bounded and $T_k$ is a bounded and Lipschitz function, by standard results, there exists a unique solution of problem \eqref{eq31} in the sense of Definition \ref{def31}. We now multiply both sides of the differential equation in problem \eqref{eq31} by $u^{q-1}$,
$$
\int_{B_R} \,u_t\,u^{q-1}\,d\mu =\int_{B_R}  \Delta( u^m)\,u^{q-1} \,d\mu\,+ \int_{B_R}  T_k(u^p)\,u^{q-1}\,d\mu \,.
$$
Now, formally integrating by parts in $B_R$. This can be justified by standard tools, by an approximation procedure. We get
\begin{equation}\label{eq43}
\frac{1}{q}\frac{d}{dt}\int_{B_R} u^{q}\,d\mu =-m(q-1)\int_{B_R}  u^{m+q-3}\,|\nabla u|^2 \,d\mu\,+ \int_{B_R}  T_k(u^p)\,u^{q-1}\,d\mu \,.
\end{equation}
Observe that, thanks to Sobolev inequality \eqref{S}, we have
\begin{equation}\label{eq43b}
\begin{aligned}
\int_{B_R}  u^{m+q-3}\,|\nabla u|^2 \,d\mu&= \frac{4}{(m+q-1)^2} \int_{B_R}\left |\nabla \left(u^{\frac{m+q-1}{2}}\right)\right|^2 \,d\mu\\
&\ge \frac{4}{(m+q-1)^2} C_s^2\left( \int_{B_R}  u^{\frac{m+q-1}{2}\frac{2N}{N-2}}\,d\mu \right)^{\frac{N-2}{N}}
\end{aligned}
\end{equation}
Moreover, the last term in the right hand side of \eqref{eq43}, thanks to H{\"o}lder inequality with exponents $\frac{N}{N-2}$ and $\frac{N}{2}$, becomes
\begin{equation}\label{eq43c}
\begin{aligned}
\int_{B_R}  T_k(u^p)\,u^{q-1}\,d\mu&\le \int_{B_R}  u^p\,u^{q-1}\,d\mu = \int_{B_R}  u^{p-m}\,u^{m+q-1}\,d\mu \\
&\le \|u(t)\|^{p-m}_{L^{(p-m)\frac{N}{2}}(B_R)} \|u(t)\|^{m+q-1}_{L^{(m+q-1)\frac{N}{N-2}}(B_R)}
\end{aligned}
\end{equation}
Combining \eqref{eq43b} and \eqref{eq43c} we get
\begin{equation}\label{eq44}
\frac{1}{q}\frac{d}{dt} \|u(t)\|^q_{L^q(B_R)}\le -\left[\frac{4\,m(q-1)}{(m+q-1)^2} C_s^2-\|u(t)\|^{p-m}_{L^{p_0}(B_R)}\right] \|u(t)\|^{m+q-1}_{L^{(m+q-1)\frac{N}{N-2}}(B_R)}\,.
\end{equation}
Take any $T>0$. Observe that, thanks to hypothesis \eqref{eq41} and the known continuity of  the map $t\mapsto u(t)$ in $[0, T]$, there exists $t_0>0$ such that
$$
\|u(t)\|_{L^{p_0}(B_R)}\le 2\, \bar\varepsilon\,\,\,\,\,\text{for any}\,\,\,\, t\in [0,t_0]\,.
$$
Hence \eqref{eq44} becomes, for any $t\in (0,t_0]$,
$$
\frac{1}{q}\frac{d}{dt} \|u(t)\|^q_{L^q(B_R)}\le -\left[\frac{4\,m(q-1)}{(m+q-1)^2} C_s^2-2\,\bar\varepsilon^{p-m} \right] \|u(t)\|^{m+q-1}_{L^{(m+q-1)\frac{N}{N-2}}(B_R)}\,\le 0\,,
$$
where the last inequality is obtained thanks to \eqref{eq41}. We have proved that $t\mapsto \|u(t)\|_{L^q(B_R)}$ is decreasing in time for any $t\in (0,t_0]$, i.e.
\begin{equation}\label{eq45}
\|u(t)\|_{L^q(B_R)}\le \|u_0\|_{L^q(B_R)}\quad \text{for any} \,\,\,t\in (0,t_0]\,.
\end{equation}
In particular, inequality \eqref{eq45} follows for the choice $q=p_0$, in view of hypothesis \eqref{eq41}. Hence we have
$$
\|u(t)\|_{L^{p_0}(B_R)}\le \|u_0\|_{L^{p_0}(B_R)}\,<\,\bar\varepsilon \quad \text{for any} \,\,\,\,t\in (0,t_0]\,.
$$
Now, we can repeat the same argument in the time interval $(t_0, t_1]$, where $t_1$ is chosen, due to the continuity of $u$, in such a way that
$$
\|u(t)\|_{L^{p_0}(B_R)}\le 2\bar\varepsilon\,\,\,\,\,\text{for any}\,\,\, t\in (t_0,t_1]\,.
$$
Thus we get
\begin{equation*}
\|u(t)\|_{L^q(B_R)}\le \|u_0\|_{L^q(B_R)}\quad \text{for any} \,\,\,t\in (0,t_1]\,.
\end{equation*}
Iterating this procedure we obtain that $t\mapsto \|u(t)\|_{L^q(B_R)}$ is decreasing in $[0, T]$. Since $T>0$ was arbitrary, the thesis follows.
\end{proof}
Using a Moser type iteration procedure we prove the following result:
\begin{proposition}\label{prop42}
Let $m>1,\, p>m+\frac{2}{N}$.
Assume that inequality \eqref{S} holds. Suppose that $u_0\in L^{\infty}(B_R)$, $u_0\ge0$. Let $u$ be the solution of problem \eqref{eq31} in the sense of Definition \ref{def31}. Let $1< q_0\le q<+\infty$ and assume that
\begin{equation}\label{epsilon0a}
\|u_0\|_{L^{p_0}(B_R)}<\tilde\varepsilon_0
\end{equation}
for $\tilde\varepsilon_0=\tilde\varepsilon_0(p, m, N, C_s, q, q_0)$ sufficiently small. Then there exists $C(m,q_0,C_s, \tilde\varepsilon_0, N, q)>0$ such that
\begin{equation*}
\|u(t)\|_{L^q(B_R)} \le C\,t^{-\gamma_q}\|u_0\|^{\delta_q}_{L^{q_0}(B_R)}\quad \textrm{ for all }\,\, t>0\,,
\end{equation*}
where
\begin{equation}\label{eq47}
\gamma_q=\left(\frac{1}{q_0}-\frac{1}{q}\right)\frac{N\,q_0}{2\,q_0+N(m-1)}\,,\quad \delta_q=\frac{q_0}{q}\left(\frac{q+\frac{N}{2}(m-1)}{q_0+\frac{N}{2}(m-1)}\right)\,.
\end{equation}
\end{proposition}

\begin{proof}
Let $\{q_n\}$ be the sequence defined in \eqref{eq400}. We start by proving a smoothing estimate from $q_0$ to $q_{\bar n}$ using a Moser iteration technique (see also \cite{A}). 

Let $t>0$, we define
\begin{equation}\label{eq48}
s=\frac{t}{2^{\overline n}-1} , \quad t_n=(2^n-1)s\,.
\end{equation}
Observe that $t_0=0, \quad t_{\bar n}=t,\quad \{t_n\}\,\text{ is an increasing sequence w.r.t.}\,\,n$.
 Now, for any $1\le n\le \overline n$, we multiply equation \eqref{eq31} by $u^{q_{n-1}-1}$ and integrate in $B_R\times[t_{n-1},t_{n}]$. Thus we get
$$
\int_{t_{n-1}}^{t_{n}}\int_{B_R} \,u_t\,u^{q_{n-1}-1}\,d\mu\,dt = \int_{t_{n-1}}^{t_{n}}\int_{B_R}  \Delta( u^m)\,u^{q_{n-1}-1} \,d\mu\,dt+ \int_{t_{n-1}}^{t_{n}}\int_{B_R}  T_k(u^p)\,u^{q_{n-1}-1}\,d\mu\,dt.
$$
Then we integrate by parts in $B_R\times[t_{n-1},t_{n}]$. Thanks to Sobolev inequality and hypothesis \eqref{epsilon0a} we get
\begin{equation}\label{eq48b}
\begin{aligned}
&\frac{1}{q_{n-1}}\left[\|u(\cdot, t_{n})\|^{q_{n-1}}_{L^{q_{n-1}}(B_R)}-\|u(\cdot, t_{n-1})\|^{q_{n-1}}_{L^{q_{n-1}}(B_R)}\right]\\&\le -\left[\frac{4m(q_{n-1}-1)}{(m+q_{n-1}-1)^2} C_s^2-2\tilde\varepsilon_0^{\frac1{p-m}}\right] \int_{t_{n-1}}^{t_{n}}\|u(\tau)\|^{m+q_{n-1}-1}_{L^{(m+q_{n-1}-1)\frac{N}{N-2}}(B_R)}\,d\tau,
\end{aligned}
\end{equation}
where we have used the fact that $T_k(u^p)\,\le\,u^p$.
We define $q_n$ as in \eqref{eq400}, so that $(m+q_{n-1}-1)\dfrac{N}{N-2}=q_{n}$. Hence, in view of hypothesis \eqref{epsilon0a} we can apply Lemma \ref{lemma41} to the integral on the right hand side of \eqref{eq48b}, hence we get
\begin{equation}\label{eq49}
\begin{aligned}
&\frac{1}{q_{n-1}}\left[\|u(\cdot, t_{n})\|^{q_{n-1}}_{L^{q_{n-1}}(B_R)}-\|u(\cdot, t_{n-1})\|^{q_{n-1}}_{L^{q_{n-1}}(B_R)}\right]\\&\le -\left[\frac{4m(q_{n-1}-1)}{(m+q_{n-1}-1)^2} C_s^2-2\tilde \varepsilon_0^{\frac1{p-m}}\right] \|u(\cdot,t_{n})\|^{m+q_{n-1}-1}_{L^{q_n}(B_R)}|t_{n}-t_{n-1}|.
\end{aligned}
\end{equation}
Observe that
\begin{equation}\label{eq410}
\begin{aligned}
&\left \|u(\cdot, t_{n})\right\|^{q_{n-1}}_{L^{q_{n-1}}(B_R)} \ge 0,\\
& |t_{n}-t_{n-1}|=\frac{2^{n-1}\,t}{2^{\bar n}-1}.
\end{aligned}
\end{equation}
We define
\begin{equation}\label{eq411}
d_{n-1}:=\left[\frac{4\,m\,(q_{n-1}-1)}{(m+q_{n-1}-1)^2}C_s^2-2\tilde\varepsilon_0^{\frac1{p-m}}\right]^{-1}\frac{1}{q_{n-1}}.
\end{equation}
By plugging \eqref{eq410} and \eqref{eq411} into \eqref{eq49} we get
$$
\|u(\cdot, t_{n})\|^{m+q_{n-1}-1}_{L^{q_n}(B_R)}\le \frac{(2^{\bar n}-1)d_n\,}{2^{n-1}\,t}\|u(\cdot,t_{n-1})\|^{q_{n-1}}_{L^{q_{n-1}}(B_R)}.
$$
The latter formula can be rewritten as
\begin{equation*}
\|u(\cdot, t_{n})\|_{L^{q_n}(B_R)}\le \left(\frac{(2^{\bar n}-1)d_n}{2^{n-1}}\right)^{\frac{1}{m+q_{n-1}-1}}\,t^{-\frac{1}{m+q_{n-1}-1}}\|u(\cdot,t_{n-1})\|^{\frac{q_{n-1}}{m+q_{n-1}-1}}_{L^{q_{n-1}}(B_R)}.
\end{equation*}
Thanks to the definition of the sequence $\{q_n\}$ in \eqref{eq400} we write
\begin{equation}\label{eq413}
\|u(\cdot, t_{n})\|_{L^{q_{n}}(B_R)}\le \left(\frac{(2^{\bar n}-1)d_{n-1}}{2^{n-1}}\right)^{\frac{N}{(N-2)}\frac{1}{q_{n}}}\,t^{-\frac{N}{(N-2)}\frac{1}{q_{n}}}\left\|u(\cdot,t_{n-1})\right\|^{\frac{q_{n-1}}{q_{n}}\frac{N}{N-2}}_{L^{q_{n-1}}(B_R)}.
\end{equation}
Define $\sigma:=\frac{N}{N-2}$.
Observe that, for any $1\le n\le \bar n$, we have
\begin{equation}\label{eq415}
\begin{aligned}
\left(\frac{(2^{\bar n}-1)d_{n-1}}{2^{n-1}}\right)^{\sigma}&= \left[\frac{2^{\bar n}-1}{2^{n-1}}\left(\frac{4\,m(q_{n-1}-1)}{(m+q_{n-1}-1)^2}C_s^2-2\varepsilon^{\frac1{p-m}}\right)^{-1}\frac{1}{q_{n-1}}\right]^{\sigma}\\
&= \left[\frac{2^{\bar n}-1}{2^{n-1}}\frac{1}{\dfrac{4\,m\,q_{n-1}(q_{n-1}-1)}{(m+q_{n-1}-1)^2}C_s^2-2\tilde\varepsilon_0^{\frac1{p-m}} q_{n-1}}\right]^{\sigma},
\end{aligned}
\end{equation}
where
\begin{equation}\label{eq416}
\frac{2^{\bar n}-1}{2^{n-1}} \le 2^{\bar n+1}\,\,\,\,\,\quad\text{for all}\,\,\, 1\le n\le \bar n.
\end{equation}
Consider the function
$$
g(x):=\left[\frac{4\,m(x-1)}{(m+x-1)^2}C_s^2-2\tilde\varepsilon_0^{\frac1{p-m}}\right] x\,\,\,\,\,\quad\text{for}\,\,\,q_0\le x \le q_{\bar n},\,\,\,x\in\R.
$$
Observe that, thanks to the definition of $\sigma$, $g(x)>0$ for any $q_0\le x \le q_{\bar n}$. Moreover, $g$ has a minimum in the interval $q_0\le x \le q_{\bar n}$, call it $\tilde x$. Then we have
\begin{equation}\label{eq417}
\frac{1}{g(x)}\le \frac{1}{g(\tilde x)} \quad\quad\text{for any }\,\,\,q_0\le x \le q_{\bar n},\,\,x\in\R.
\end{equation}
Thanks to \eqref{eq415}, \eqref{eq416} and \eqref{eq417}, we can say that there exist a positive constant $C$, where $C=C(N,C_s,\varepsilon, \bar n,m,q_0)$, such that
\begin{equation}\label{eq418}
\left(\frac{(2^{\bar n}-1)d_{n-1}}{2^{n-1}}\right)^{\sigma} \le C\,,\quad  \text{for all}\,\,\, 1\le n\le\bar n.
\end{equation}
By using \eqref{eq418} and \eqref{eq413} we get, for any $1\le n\le \bar n$
\begin{equation}\label{eq419}
\|u(\cdot, t_{n})\|_{L^{q_{n}}(B_R)}\le C^{\frac{1}{q_n}}t^{-\frac{\sigma}{q_{n}}}\left\|u(\cdot,t_{n-1})\right\|^{\frac{q_{n-1}\sigma}{q_{n}}}_{L^{q_{n-1}}(B_R)}.
\end{equation}
Let us set
$$
U_n:=\|u(\cdot,t_n)\|_{L^{q_n}(B_R)}.
$$
Then \eqref{eq419} becomes
$$
\begin{aligned}
U_n&\le C^{\frac{1}{q_n}}t^{-\frac{\sigma}{q_{n}}}U_{n-1}^{\frac{q_{n-1}\sigma}{q_{n}}}\\
&\le C^{\frac{1}{q_n}}t^{-\frac{\sigma}{q_{n}}}\left [ C^{\frac{\sigma}{q_n}}t^{-\frac{\sigma^2}{q_{n}}} U_{k-2}^{\sigma^2\frac{q_{n-2}}{q_n}}\right] \\
&\le ...\\
&\le C^{\frac{1}{q_n}\sum_{i=0}^{n-1}\sigma^i}t^{-\frac{\sigma}{q_n}\sum_{i=0}^{n-1}\sigma^i} U_0^{\sigma^n\frac{q_0}{q_n}}.
\end{aligned}
$$
We define
\begin{equation}\label{eq420}
\begin{aligned}
&\alpha_n:= \frac{1}{q_n}\sum_{i=0}^{n-1}\sigma^i,\ \ \beta_n:= \frac{\sigma}{q_n}\sum_{i=0}^{n-1}\sigma^i=\sigma\,\alpha_n,
\ \ \delta_n:=\sigma^n\frac{q_0}{q_n}.
\end{aligned}
\end{equation}
By substituting $n$ with $\bar n$ into \eqref{eq420} we get
\begin{equation}\label{eq421}
\begin{aligned}
\alpha_{\bar n}:=\frac{N-2}{2}\frac{A}{q_{\bar n}},\ \ \beta_{\bar n}:=\frac{N}{2}\frac{A}{ q_{\bar n}},\ \
\delta_{\bar n}:=(A+1)\frac{q_0}{q_{\bar n}}.
\end{aligned}
\end{equation}
where $A:=\left(\frac{N}{N-2}\right)^{\bar n}-1$. Hence, in view of \eqref{eq48} and \eqref{eq421}, \eqref{eq419} with $n=\bar n$ yields
\begin{equation}\label{eq422b}
\|u(\cdot, t)\|_{L^{q_{\bar n}}(B_R)}\le C^{\frac{N-2}{2}\frac{A}{q_{\bar n}}}\,t^{-\frac{N}{2}\frac{A}{q_{\bar n}}}\left\|u_0\right\|^{q_{0}\frac{A+1}{q_{\bar n}}}_{L^{q_{0}}(B_R)}.
\end{equation}
We have proved a smoothing estimate from $q_0$ to $q_{\bar n}$.
Observe that if $q_{\bar n}= q$ then the thesis is proved.
Now suppose that $q>q_{\bar n}$. Observe that $q_0\le  q < q_{\bar n}$ and define
$$
B:=N(m-1)A+2\,q_0(A+1).
$$
From \eqref{eq422b} and Lemma \ref{lemma41} we get, by interpolation,
\begin{equation}\label{eq423b}
\begin{aligned}
\|u(\cdot, t)\|_{L^{ q}(B_R)}&\le \|u(\cdot, t)\|_{L^{q_0}(B_R)}^{\theta}\|u(\cdot, t)\|_{L^{ q_{\bar n}}(B_R)}^{1-\theta}\\
&\le \|u_0(\cdot)\|_{L^{q_0}(B_R)}^{\theta} C\,t^{-\frac{N\,A}{B}(1-\theta)}\left\|u_0\right\|^{2q_{0}\frac{A+1}{B}(1-\theta)}_{L^{q_{0}}(B_R)}\\
&=C\,t^{-\frac{N\,A}{B}(1-\theta)}\left\|u_0\right\|^{2q_{0}\frac{A+1}{B}(1-\theta)+\theta}_{L^{q_{0}}(B_R)},
\end{aligned}
\end{equation}
where
\begin{equation}\label{eq424b}
\theta=\frac{q_0}{ q}\left(\frac{q_{\bar n}- q}{q_{\bar n}-q_0}\right).
\end{equation}
Combining \eqref{eq423b}, \eqref{eq47} and \eqref{eq424b} we get the claim, noticing that $q$ was arbitrary in  $[q_0, \infty)$.
\end{proof}

\begin{remark}
One can not let $q\to+\infty$ in the above bound. In fact, one can show that $\varepsilon \longrightarrow 0\ \text{as}\ q\to\infty.$ So in such limit the hypothesis on the norm of the initial datum \eqref{epsilon0} is satisfied only when $u_0\equiv 0$.
\end{remark}

\begin{proposition}\label{prop44}
Let $m>1$, $p>m+\frac{2}{N}$, $R>0$, $p_0$ be as in \eqref{p0}, $u_0\in  L^{\infty}(B_R)$, $u_0\ge 0$. Let
\begin{equation}\label{eq424}
\begin{aligned}
&r>\,\max\left\{p_0,\, \frac N2\right\},\quad\quad s=1+\frac 2N-\frac 1r.
\end{aligned}
\end{equation}
Suppose that \eqref{epsilon0} holds for $\varepsilon_0=\varepsilon_0(p, m, N, C_s, r)$ sufficiently small.
Let $u$ be the solution to problem \eqref{eq31}. Let $M$ be such that inequality \eqref{S} holds.
Then there exists $\Gamma=\Gamma(p, m, N, r)>0$ such that, for all $t>0$,
\begin{equation}\label{eq422}
\|u(t)\|_{L^{\infty}(B_R)}
\le \Gamma\, t^{-\frac{\gamma}{ms}}\left\{\|u_0\|_{L^{p_0}(B_R)}^{\delta_{1}}+\frac{1}{m-1}\,\|u_0\|_{L^{p_0}(B_R)}^{\delta_{2}} \right\}^{\frac{1}{ms}}\|u_0\|_{L^{m}(B_R)}^{\frac{s-1}{s}},
\end{equation}
where
\begin{equation}\label{eq423}
\gamma= \frac{p}{p-1}\left[1-\frac{N(p-m)}{2\,p\,r}\right],\,\delta_{1}=p\,\frac{p-m}{m-1}\left[1+\frac{N(m-1)}{2\,p\,r}\right],\, \delta_{2}=\frac{p-m}{m-1}\left[1+\frac{N(m-1)}{2\,r}\right]
\end{equation}
\end{proposition}

\begin{remark}\label{remark3}
\rm If in Proposition \ref{prop44}, in addition, we assume that for some $k_0>0$
$$
\max\left\{\|u_0\|_{L^m(B_R)};\,\,\|u_0\|_{L^{p_0}(B_R)}\right\}\leq k_0\quad \textrm{ for every }\,\, R>0\,,
$$
then the bound from above for $\|u(t)\|_{L^{\infty}(B_R)}$ in \eqref{eq422} is independent of $R$.
\end{remark}

\begin{proof}[Proof of Proposition \ref{prop44}]
Let us set $w=u(\cdot,t)$. Observe that $w^m\in H_0^1(B_R)$ and $w\ge0$. Due to Proposition \ref{prop43} we know that
\begin{equation*}
-\Delta(w^m) \le \left [w^p+\frac{w}{(m-1)t} \right].
\end{equation*}
Observe that, since $u_0\in L^{\infty}(B_R)$ also $w\in L^{\infty}(B_R)$.
Due to \eqref{eq424}, we can apply Proposition \ref{prop1}.
So, we have that
\begin{equation}\label{eq425}
\begin{aligned}
\|w\|_{L^{\infty}(B_R)}^m&\le \frac{s}{s-1}\left(\frac{1}{C_s}\right)^{\frac 2s}\left\|w^p+\frac{w}{(m-1)t}\right\|_{L^{r}(B_R)}^{\frac{1}{s}} \|w^m\|_{L^{1}(B_R)}^{\frac{s-1}{s}}\\
&\le \frac{s}{s-1}\left(\frac{1}{C_s}\right)^{\frac 2s}\left\{\left\|w^p\right\|_{L^{r}(B_R)}+\frac{1}{(m-1)t}\left\|w\right\|_{L^{r}(B_R)}\right\}^{\frac{1}{s}} \|w\|_{L^{m}(B_R)}^{m\frac{s-1}{s}}
\end{aligned}
\end{equation}
where $s$ has been defined in \eqref{eq35}. Thanks to \eqref{epsilon0}, with an appropriate choice of $\varepsilon_0$, and \eqref{eq424} we can apply Proposition \ref{prop42} with
$$
q=pr,\quad q_0=p_0, \quad \gamma_{pr}=\frac{1}{p-1}\left[1-\frac{N(p-m)}{2pr}\right]$$
and $\delta_{pr}=\delta_1/p$, $\delta_1$ defined in \eqref{eq423}.
Hence we obtain
\begin{equation}\label{eq426}
\|w^p\|_{L^{r}(B_R)} = \left\|w\right\|^p_{L^{pr}(B_R)} \le \left[C\, t^{-\gamma_{pr}}\|u_0\|_{L^{p_0}(B_R)}^{\delta_{1}/p}\right]^{p},
\end{equation}
where $C>0$ is defined in Proposition \ref{prop42}. Similarly, by \eqref{epsilon0}, with an appropriate choice of $\varepsilon_0$, and \eqref{eq424}, we can apply Proposition \ref{prop42} with
$$
q=r,\quad q_0=p_0, \quad \gamma_{r}=\frac{1}{p-1}\left[1-\frac{N(p-m)}{2r}\right]$$
and $\delta_{r}=\delta_2$ as defined in \eqref{eq423}.
Hence we obtain
\begin{equation}\label{eq427}
\|w\|_{L^{r}(B_R)} \le C t^{-\gamma_{r}}\|u_0\|_{L^{p_0}(B_R)}^{\delta_{2}},
\end{equation}
where $C>0$ is defined in Proposition \ref{prop42}. Plugging \eqref{eq426} and \eqref{eq427} into \eqref{eq425} we obtain
$$
\begin{aligned}
\|w\|^m_{L^{\infty}(B_R)} &\le \frac{s}{s-1}\left(\frac{1}{C_s}\right)^{\frac 2s}\left\{\left\|w^p\right\|_{L^{r}(B_R)}+\frac{1}{(m-1)t}\left\|w\right\|_{L^{r}(B_R)}\right\}^{\frac{1}{s}} \|w\|_{L^{m}(B_R)}^{m\frac{s-1}{s}}\\
&\le\frac{s}{s-1}\left(\frac{1}{C_s}\right)^{\frac 2s}\left\{C^p\,t^{-p\,\gamma_{pr}}\|u_0\|_{L^{p_0}(B_R)}^{\delta_{1}}+\frac{1}{(m-1)t}\,C\,t^{-\gamma_{r}}\|u_0\|_{L^{p_0}(B_R)}^{\delta_{2}} \right\}^{\frac{1}{s}}\|w\|_{L^{m}(B_R)}^{m\frac{s-1}{s}}.
\end{aligned}
$$
Observe that $-p\gamma_{pr}=-\gamma_r-1=\gamma,$ where $\gamma$ has been defined in \eqref{eq423}. Hence we obtain
\begin{equation*}
\|w\|^m_{L^{\infty}(B_R)}
\le  \frac{s}{s-1}\left(\frac{1}{C_s}\right)^{\frac 2s} t^{-\frac{\gamma}{s}}\left\{C^p\,\|u_0\|_{L^{p_0}(B_R)}^{\delta_{1}}+\frac{1}{m-1}\,C\,\|u_0\|_{L^{p_0}(B_R)}^{\delta_{2}} \right\}^{\frac{1}{s}}\|w\|_{L^{m}(B_R)}^{m\frac{s-1}{s}}.
\end{equation*}
Moreover, since $u_0\in L^{\infty}(B_R)$, we can apply Lemma \ref{lemma41} to $w$ with $q=m$. Thus from \eqref{eq42} with $q=m$ we get
\begin{equation*}
\|w\|^m_{L^{\infty}(B_R)}
\le \frac{s}{s-1} \left(\frac{1}{C_s}\right)^{\frac 2s}t^{-\frac{\gamma}{s}}\left\{C^p\,\|u_0\|_{L^{p_0}(B_R)}^{\delta_{1}}+\frac{1}{m-1}\,C\,\|u_0\|_{L^{p_0}(B_R)}^{\delta_{2}} \right\}^{\frac{1}{s}}\|u_0\|_{L^{m}(B_R)}^{m\frac{s-1}{s}}.
\end{equation*}
Finally define
\begin{equation*}
\Gamma:= \left[\frac{s}{s-1}\left(\frac{1}{C_s}\right)^{\frac 2s}\,\max\left\{C^{\frac{p}{s}}\,;\,\,\,C^{\frac{1}{s}}\right\}\right]^{\frac 1m}.
\end{equation*}
Hence we obtain
\begin{equation*}
\|w\|_{L^{\infty}(B_R)}
\le \Gamma\, t^{-\frac{\gamma}{ms}}\left\{\|u_0\|_{L^{p_0}(B_R)}^{\delta_{1}}+\frac{1}{m-1}\,\|u_0\|_{L^{p_0}(B_R)}^{\delta_{2}} \right\}^{\frac{1}{ms}}\|u_0\|_{L^{m}(B_R)}^{\frac{s-1}{s}}.
\end{equation*}

\end{proof}

\section{Proof of Theorem \ref{teo22}}\label{proofs}

\begin{proof}[Proof of Theorem \ref{teo22}]
Let $\{u_{0,h}\}_{h\ge 0}$ be a sequence of functions such that
\begin{equation*}
\begin{aligned}
&(a)\,\,u_{0,h}\in L^{\infty}(M)\cap C_c^{\infty}(M) \,\,\,\text{for all} \,\,h\ge 0, \\
&(b)\,\,u_{0,h}\ge 0 \,\,\,\text{for all} \,\,h\ge 0, \\
&(c)\,\,u_{0, h_1}\leq u_{0, h_2}\,\,\,\text{for any } h_1<h_2,  \\
&(d)\,\,u_{0,h}\longrightarrow u_0 \,\,\, \text{in}\,\, L^m(M)\cap L^{p_0}(M)\quad \textrm{ as }\, h\to +\infty\,,\\
\end{aligned}
\end{equation*}
where $p_0$ has been defined in \eqref{p0}. Observe that, due to assumptions $(c)$ and $(d)$, $u_{0,h}$ satisfies \eqref{epsilon0}. For any $R>0$, $k>0$, $h>0$, consider the problem
\begin{equation}\label{5}
\begin{cases}
u_t= \Delta u^m +T_k(u^p) &\text{in}\,\, B_R\times (0,+\infty)\\
u=0& \text{in}\,\, \partial B_R\times (0,\infty)\\
u=u_{0,h} &\text{in}\,\, B_R\times \{0\}\,. \\
\end{cases}
\end{equation}
From standard results it follows that problem \eqref{5} has a solution $u_{h,k}^R$ in the sense of Definition \ref{def31}; moreover, $u^R_{h,k}\in C\big([0, T]; L^q(B_R)\big)$ for any $q>1$. Hence, by Lemma \ref{lemma41}, in Proposition \ref{prop42} and in Proposition \ref{prop44}, we have for any $t\in(0,+\infty)$,
\begin{equation}\label{eq52}
\|u_{h,k}^R(t)\|_{L^m(B_R)}\,\le\, \|u_{0,h}\|_{L^m(B_R)};
\end{equation}
\begin{equation}\label{eq53}
\|u_{h,k}^R(t)\|_{L^p(B_R)}\le C\,t^{-\gamma_p} \|u_{0,h}\|^{\delta_p}_{L^{p_0}(B_R)}\,;
\end{equation}
where
$$
\gamma_p=\frac{1}{p-1}\left[1-\frac{N(p-m)}{2p}\right],\quad \delta_p=\frac{p-m}{p-1}\left[1+\frac{N(m-1)}{2p}\right]\,,
$$
\begin{equation}\label{eq54}
\|u_{h,k}^R\|_{L^{\infty}(B_R)}
\le \Gamma\, t^{-\frac{\gamma}{ms}}\left\{\|u_{0,h}\|_{L^{p_0}(B_R)}^{\delta_{1}}+\frac{1}{m-1}\,\|u_{0,h}\|_{L^{p_0}(B_R)}^{\delta_{2}} \right\}^{\frac{1}{ms}}\|u_{0,h}\|_{L^{m}(B_R)}^{\frac{s-1}{s}},
\end{equation}
with $s$ as in \eqref{eq424} and $\gamma$, $\delta_1$, $\delta_2$ as in \eqref{eq423}. In addition, for any $\tau\in (0, T), \zeta\in C^1_c((\tau, T)), \zeta\geq 0$, $\max_{[\tau, T]}\zeta'>0$,
\begin{equation}\label{eqcont1}
\begin{aligned}
\int_{\tau}^T \zeta(t) \left[\big((u^R_{h, k})^{\frac{m+1}2}\big)_t\right]^2 d\mu dt &\leq \max_{[\tau, T]}\zeta' \bar C \int_{B_R}(u_{h, k}^R)^{m+1}(x, \tau)d\mu\\
&+ \bar C \max_{[\tau, T]}\zeta \int_{B_R} F\big(u^{R}_{h, k}(x,T)\big)d\mu\\
&\leq  \max_{[\tau, T]}\zeta'(t)\bar C \|u^R_{h, k}(\tau)\|_{L^\infty(B_R)}\|u^R_{h, k}(\tau)\|_{L^m(B_R)}^m \\
&+\frac {\bar C}{m+p}\|u^R_{h, k}(T)\|^p_{L^\infty(B_R)}\|u^R_{h, k}(T)\|_{L^m(B_R)}^m
\end{aligned}
\end{equation}
where
\[F(u)=\int_0^u s^{m-1+p} \, ds\,,\]
and $\bar C>0$ is a constant only depending on $m$. Inequality \eqref{eqcont1} is formally obtained by multiplying the differential inequality in problem \eqref{eq31} by $\zeta(t)[(u^m)_t]$, and integrating by parts; indeed, a standard approximation procedure is needed (see \cite[Lemma 3.3]{GMPo} and \cite[Theorem 13]{ACP}).

\smallskip

Moreover, as a consequence of Definition \ref{def31}, for any $\varphi \in C_c^{\infty}(B_R\times[0,T])$ such that $\varphi(x,T)=0$ for any $x\in B_R$, $u_{h,k}^R$ satisfies
\begin{equation}\label{eq56}
\begin{aligned}
-\int_0^T\int_{B_R}u_{h,k}^R\,\varphi_t\,d\mu\,dt =&\int_0^T\int_{B_R} (u_{h,k}^R)^m\,\Delta\varphi\,d\mu\,dt\,+ \int_0^T\int_{B_R} T_k[(u_{h,k}^R)^p]\,\varphi\,d\mu\,dt \\
& +\int_{B_R} u_{0,h}(x)\,\varphi(x,0)\,d\mu,
\end{aligned}
\end{equation}
where all the integrals are finite.
Now, observe that, for any $h>0$ and $R>0$ the sequence of solutions $\{u_{h,k}^R\}_{k\ge0}$ is monotone increasing in $k$ hence it has a pointwise limit for $k\to \infty$. Let $u_h^R$ be such limit so that we have
$$
u_{h,k}^R\longrightarrow u_{h}^R \quad \text{as} \,\,\, k\to\infty \,\,\text{pointwise}.
$$
In view of \eqref{eq52}, \eqref{eq53} and \eqref{eq54}, the right hand side of \eqref{eqcont1} is independent of $k$. So, $(u^R_h)^{\frac{m+1}2}\in H^1((\tau, T); L^2(B_R))$. Therefore, $(u^R_h)^{\frac{m+1}2}\in C\big([\tau, T]; L^2(B_R)\big)$. We can now pass to the limit as $k\to +\infty$ in inequalities \eqref{eq52}, \eqref{eq53} and \eqref{eq54} arguing as follows. From inequality \eqref{eq52} and \eqref{eq53}, thanks to the Fatou's Lemma, one has for all $t>0$
\begin{equation}\label{eq58}
\begin{aligned}
\|u_{h}^R(t)\|_{L^m(B_R)}\leq  \|u_{0,h}\|_{L^m(B_R)}.
\end{aligned}
\end{equation}
\begin{equation}\label{eq59}
\|u_{h}^R(t)\|_{L^p(B_R)}\le C\,t^{-\gamma_p} \|u_{0,h}\|^{\delta_p}_{L^{p_0}(B_R)}\,;
\end{equation}
On the other hand, from \eqref{eq54}, since $u_{h,k}^R\longrightarrow u_{h}^R$ as $k\to \infty$ pointwise and the right hand side of \eqref{eq54} is independent of $k$, one has for all $t>0$
\begin{equation}\label{eq510}
\|u_{h}^R\|_{L^{\infty}(B_R)}
\le \Gamma\, t^{-\frac{\gamma}{ms}}\left\{\|u_{0,h}\|_{L^{p_0}(B_R)}^{\delta_{1}}+\frac{1}{m-1}\,\|u_{0,h}\|_{L^{p_0}(B_R)}^{\delta_{2}} \right\}^{\frac{1}{ms}}\|u_{0,h}\|_{L^{m}(B_R)}^{\frac{s-1}{s}},
\end{equation}
with $s$ as in \eqref{eq424} and $\gamma$, $\delta_1$, $\delta_2$ as in \eqref{eq423}. Note that \eqref{eq58}, \eqref{eq59} and \eqref{eq510} hold {\em for all} $t>0$, in view of the continuity property of $u$ deduced above.
Moreover, thanks to Beppo Levi's monotone convergence theorem, it is possible to compute the limit as $k\to +\infty$ in the integrals of equality \eqref{eq56} and hence obtain that, for any $\varphi \in C_c^{\infty}(B_R\times(0,T))$ such that $\varphi(x,T)=0$ for any $x\in B_R$, the function $u_h^R$ satisfies
\begin{equation}\label{eq511}
\begin{aligned}
-\int_0^T\int_{B_R} u_{h}^R\,\varphi_t\,d\mu\,dt =&\int_0^T\int_{B_R} \left(u_{h}^R\right)^m\,\Delta\varphi\,d\mu\,dt+ \int_0^T\int_{B_R} \left(u_{h}^R\right)^p\,\varphi\,d\mu\,dt \\
& +\int_{B_R} u_{0,h}(x)\,\varphi(x,0)\,d\mu.
\end{aligned}
\end{equation}
Observe that all the integrals in \eqref{eq511} are finite, hence $u_h^R$ is a solution to problem \eqref{5}, where we replace $T_k(u^p)$ with $u^p$ itself, in the sense of Definition \ref{def31}. Indeed we have, due to \eqref{eq58}, $u_{h}^R \in L^m(B_R\times(0,T))$ hence $u_{h}^R \in L^1(B_R\times(0,T))$. Moreover, due to \eqref{eq59}, $u_{h}^R \in L^p(B_R\times(0,T))$ indeed we can write
\begin{equation}\label{eq512}
\begin{aligned}
\int_0^T\int_{B_R} \left(u_{h}^R\right)^p\,d\mu\,dt\, &=\int_0^T\|u_{h}^R\|^p_{L^p(B_R)}\,dt\\
&\le \int_0^T \left(C\,t^{-\gamma_p} \|u_{0,h}\|^{\delta_p}_{L^{p_0}(B_R)}\right)^p\,dt\\
&= C^p\,\|u_{0,h}\|^{p\delta_p}_{L^{p_0}(B_R)}\int_0^T t^{-p\gamma_p}\,dt.
\end{aligned}
\end{equation}
Now observe that the integral in \eqref{eq512} is finite if and only if $p\,\gamma_p<1\,.$
The latter reads $p>m+\frac{2}{N}$, which is guaranteed by the hypotheses of Theorem \ref{teo22}.

Let us now observe that, for any $h>0$, the sequence of solutions $\{u_h^R\}_{R>0}$ is monotone increasing in $R$, hence it has a pointwise limit as $R\to+\infty$. We call its limit function $u_h$ so that
$$
u_{h}^R\longrightarrow u_{h} \quad \text{as} \,\,\, R\to+\infty \,\,\text{pointwise}.
$$
In view of \eqref{eq52}, \eqref{eq53}, \eqref{eq54}, \eqref{eq58}, \eqref{eq59}, \eqref{eq510}, the right hand side of \eqref{eqcont1} is independent of $k$ and $R$. So, $(u_h)^{\frac{m+1}2}\in H^1((\tau, T); L^2(M))$. Therefore, $(u_h)^{\frac{m+1}2}\in C\big([\tau, T]; L^2(M)\big)$. Since $u_0\in L^m(M)\cap L^{p_0}(M)$, there exists $k_0>0$ and $k_1>0$ such that
\begin{equation}\label{eq513}
\begin{aligned}
&\|u_{0h}\|_{L^m(B_R)}\leq k_0 \quad\quad\quad\, \forall\,\, h>0,\,\,\,\, \forall\,\,R>0\,,\\
&\|u_{0h}\|_{L^{p_0}(B_R)}\leq k_1 \quad \forall\,\, h>0,\,\,\,\, \forall\,\,R>0\,.
\end{aligned}
\end{equation}
Note that, in view of \eqref{eq513}, the norms in  \eqref{eq58}, \eqref{eq59} and \eqref{eq510} do not depend on $R$ (see Lemma \ref{lemma41}, Proposition \ref{prop42}, Proposition \ref{prop44} and Remark \ref{remark3}). Therefore, we pass to the limit as $R\to+\infty$ in \eqref{eq58}, \eqref{eq59} and \eqref{eq510}. By Fatou's Lemma,
\begin{equation}\label{eq514}
\begin{aligned}
\|u_{h}(t)\|_{L^m(M)}\leq  \|u_{0,h}\|_{L^m(M)},
\end{aligned}
\end{equation}
\begin{equation}\label{eq515}
\|u_{h}(t)\|_{L^p(M)}\le C\,t^{-\gamma_p} \|u_{0,h}\|^{\delta_p}_{L^{p_0}(M)}\,,
\end{equation}
furthermore, since $u_{h}^R\longrightarrow u_{h} $ as $R\to +\infty$ pointwise,
\begin{equation}\label{eq516}
\|u_{h}\|_{L^{\infty}(M)}
\le \Gamma\, t^{-\frac{\gamma}{ms}}\left\{\|u_{0,h}\|_{L^{p_0}(M)}^{\delta_{1}}+\frac{1}{m-1}\,\|u_{0,h}\|_{L^{p_0}(M)}^{\delta_{2}} \right\}^{\frac{1}{ms}}\|u_{0,h}\|_{L^{m}(M)}^{\frac{s-1}{s}},
\end{equation}
with $s$ as in \eqref{eq424} and $\gamma$, $\delta_1$, $\delta_2$ as in \eqref{eq423}.
Note that \eqref{eq514}, \eqref{eq515} and \eqref{eq516} hold {\em for all} $t>0$, in view of the continuity property of $u^R_h$ deduced above.

Moreover, again by monotone convergence, it is possible to compute the limit as $R\to +\infty$ in the integrals of equality \eqref{eq511} and hence obtain that, for any $\varphi \in C_c^{\infty}(M\times(0,T))$ such that $\varphi(x,T)=0$ for any $x\in M$, the function $u_h$ satisfies,
\begin{equation}\label{eq517}
\begin{aligned}
-\int_0^T\int_{M} u_{h}\,\varphi_t\,d\mu\,dt =&\int_0^T\int_{M} (u_{h})^m\,\Delta\varphi\,d\mu\,dt+ \int_0^T\int_{M} (u_{h})^p\,\varphi\,d\mu\,dt \\
& +\int_{M} u_{0,h}(x)\,\varphi(x,0)\,d\mu.
\end{aligned}
\end{equation}
Observe that, arguing as above, due to inequalities \eqref{eq514} and \eqref{eq515}, all the integrals in \eqref{eq517} are well posed hence $u_h$ is a solution to problem \eqref{problema}, where we replace $u_0$ with $u_{0,h}$, in the sense of Definition \ref{def21}.
Finally, let us observe that $\{u_{0,h}\}_{h\ge0}$ has been chosen in such a way that
$$
u_{0,h}\longrightarrow u_0 \,\,\, \text{in}\,\, L^m(M)\cap L^{p_0}(M).
$$
Observe also that $\{u_{h}\}_{h\ge0}$ is a monotone increasing function in $h$ hence it has a limit as $h\to+\infty$. We call $u$ the limit function.
In view  \eqref{eq52}, \eqref{eq53}, \eqref{eq54}, \eqref{eq58}, \eqref{eq59}, \eqref{eq510}, \eqref{eq514}, \eqref{eq515} and \eqref{eq516}  the right hand side of \eqref{eqcont1} is independent of $k, R$ and $h$. So, $u^{\frac{m+1}2}\in H^1((\tau, T); L^2(M))$. Therefore, $u^{\frac{m+1}2}\in C\big([\tau, T]; L^2(M)\big)$. Hence, we can pass to the limit as $h\to +\infty$ in \eqref{eq514}, \eqref{eq515} and \eqref{eq516} and similarly to what we have seen above, we get
\begin{equation}\label{eq518}
\begin{aligned}
\|u(t)\|_{L^m(M)}\leq  \|u_{0}\|_{L^m(M)},
\end{aligned}
\end{equation}
\begin{equation}\label{eq519}
\|u(t)\|_{L^p(M)}\le C\,t^{-\gamma_p} \|u_{0}\|^{\delta_p}_{L^{p_0}(M)}\,,
\end{equation}
and
\begin{equation}\label{eq520}
\|u\|_{L^{\infty}(M)}
\le \Gamma\, t^{-\frac{\gamma}{ms}}\left\{\|u_0\|_{L^{p_0}(M)}^{\delta_{1}}+\frac{1}{m-1}\,\|u_0\|_{L^{p_0}(M)}^{\delta_{2}} \right\}^{\frac{1}{ms}}\|u_0\|_{L^{m}(M)}^{\frac{s-1}{s}},
\end{equation}
with $s$ as in \eqref{eq424} and $\gamma$, $\delta_1$, $\delta_2$ as in \eqref{eq423}.
Note that both \eqref{eq518}, \eqref{eq519} and \eqref{eq520} hold {\em for all} $t>0$, in view of the continuity property of $u$ deduced above.

Moreover, again by monotone convergence, it is possible to compute the limit as $h\to+\infty$ in the integrals of equality \eqref{eq517} and hence obtain that, for any $\varphi \in C_c^{\infty}(M\times(0,T))$ such that $\varphi(x,T)=0$ for any $x\in M$, the function $u$ satisfies,
\begin{equation}\label{eq521}
\begin{aligned}
-\int_0^T\int_{M} u\,\varphi_t\,d\mu\,dt =&\int_0^T\int_{M} u^m\,\Delta\varphi\,d\mu\,dt+ \int_0^T\int_{M} u^p\,\varphi\,d\mu\,dt \\
& +\int_{M} u_{0}(x)\,\varphi(x,0)\,d\mu.
\end{aligned}
\end{equation}
Observe that, due to inequalities \eqref{eq518} and \eqref{eq519}, all the integrals in \eqref{eq521} are finite, hence $u$ is a solution to problem \eqref{problema} in the sense of Definition \ref{def21}.

\smallskip

\medskip

Finally, let us discuss \eqref{eq22} and \eqref{eq23}. Let $p_0\le q<\infty$, and observe that, thanks to hypotheses $(c)$ and $(d)$, $u_{0h}$ satisfies hypothesis \eqref{eps3a} for such $q$ and $q_0=p_0$ as $u_0$, then we have
\begin{equation}\label{eq523}
\|u_{h,k}^R(t)\|_{L^q(B_R)}\,\le\, C \,t^{-\gamma_q}\|u_{0,h}\|^{\delta_q}_{L^{p_0}(B_R)}.
\end{equation}
Hence, due to \eqref{eq523}, letting $k\to +\infty$, $R\to +\infty$, $h\to +\infty$, by Fatou's Lemma we deduce \eqref{eq23}.

Now let $1<q<\infty$. If $u_0\in L^q(M)\cap L^m(M)\cap L^{p_0}(M)$, we choose the sequence $u_{0h}$ in such a way that it further satisfies
\[u_{0,h}\longrightarrow u_0 \quad \textrm{ in }\,\, L^q(M)\,\quad \textrm{ as }\, h\to +\infty\,,\]
and observe that $u_{0h}$ satisfies also \eqref{epsilon1} for such $q$.
Then we have that
\begin{equation}\label{eq522}
\|u_{h,k}^R(t)\|_{L^q(B_R)}\,\le\, \|u_{0,h}\|_{L^q(B_R)}.
\end{equation}
Hence, due to \eqref{eq522}, letting $k\to +\infty$, $R\to +\infty$, $h\to +\infty$, by Fatou's Lemma we deduce \eqref{eq22}.
\end{proof}

\section{Estimates for $p>m$}\label{Lpbis}

\begin{lemma}\label{lemma71}
Let $m>1, p>m$. Assume that inequalities \eqref{P} and \eqref{S} hold. Suppose that $u_0\in L^{\infty}(B_R)$, $u_0\ge0$. Let $1<q<\infty$ and assume that
\begin{equation}\label{epsilon11a}
\|u_0\|_{L^{p\frac N2}(B_R)}<\tilde\varepsilon_1
\end{equation}
for a suitable $\tilde\varepsilon_1=\tilde \varepsilon_1(p, m, N, C_p, C_s, q)$ sufficiently small. Let $u$ be the solution of problem \eqref{eq31} in the sense of Definition \ref{def31},
such that in addition $u\in C([0, T); L^q(B_R))$. Then
\begin{equation}\label{eq72}
\|u(t)\|_{L^q(B_R)} \le \|u_0\|_{L^q(B_R)}\quad \textrm{ for all }\,\, t>0\,.
\end{equation}
\end{lemma}

\begin{proof}
Since $u_0$ is bounded and $T_k$ is a bounded and Lipschitz function, by standard results, there exists a unique solution of problem \eqref{eq31} in the sense of Definition \ref{def31}. We now multiply both sides of the differential equation in problem \eqref{eq31} by $u^{q-1}$, therefore
$$
\int_{B_R} \,u_t\,u^{q-1}\,d\mu =\int_{B_R}  \Delta( u^m)\,u^{q-1} \,d\mu\,+ \int_{B_R}  T_k(u^p)\,u^{q-1}\,d\mu \,.
$$
We integrate by parts. This can be justified by standard tools, by an approximation procedure. Using the fact that $T(u^p)\le u^p$, we can write
\begin{equation}\label{eq73}
\begin{aligned}
\frac{1}{q}\frac{d}{dt}\int_{B_R} u^{q}\,d\mu &\le-m(q-1)\int_{B_R}  u^{m+q-3}\,|\nabla u|^2 \,d\mu\,+ \int_{B_R}  u^p\,u^{q-1}\,d\mu \,\\
&\le-\frac{4m(q-1)}{(m+q-1)^2}\int_{B_R} \left|\nabla \left(u^{\frac{m+q-1}{2}}\right)\right|^2 \,d\mu\,+ \int_{B_R}  u^{p+q-1}\,d\mu.
\end{aligned}
\end{equation}
Now we take $c_1>0$, $c_2>0$ such that $c_1+c_2=1$. Thus
\begin{equation}\label{eq73b}
\int_{B_R} \left|\nabla \left(u^{\frac{m+q-1}{2}}\right)\right|^2 \,d\mu = c_1\, \left\|\nabla \left(u^{\frac{m+q-1}{2}}\right)\right\|_{L^2(B_R)}^2 \, + c_2\, \left\|\nabla \left(u^{\frac{m+q-1}{2}}\right)\right\|_{L^2(B_R)}^2.
\end{equation}
Take any $\alpha\in (0,1).$ Thanks to \eqref{P}, \eqref{eq73b} becomes
\begin{equation}\label{eq74}
\begin{aligned}
\int_{B_R} \left|\nabla \left(u^{\frac{m+q-1}{2}}\right)\right|^2 \,d\mu& \ge  c_1\,C_p^2 \left\| u\right\|^{m+q-1}_{L^{m+q-1}(B_R)}\, + c_2\, \left\|\nabla \left(u^{\frac{m+q-1}{2}}\right)\right\|_{L^2(B_R)}^2 \\
&\ge c_1\,C_p^2 \left\| u\right\|^{m+q-1}_{L^{m+q-1}(B_R)}\, +c_2\, \left\|\nabla \left(u^{\frac{m+q-1}{2}}\right)\right\|_{L^2(B_R)}^{2+2\alpha-2\alpha}\\
&\ge c_1C_p^2 \left\| u\right\|^{m+q-1}_{L^{m+q-1}(B_R)}+ c_2C_p^{2\alpha} \left\| u\right\|^{\alpha(m+q-1)}_{L^{m+q-1}(B_R)} \left\|\nabla \left(u^{\frac{m+q-1}{2}}\right)\right\|_{L^2(B_R)}^{2-2\alpha}\,.
\end{aligned}
\end{equation}
Moreover, using the interpolation inequality, H\"older inequality and \eqref{S}, we have
\begin{equation}\label{eq75}
\begin{aligned}
\int_{B_R}  u^{p+q-1}\,d\mu,&=\|u\|_{L^{p+q-1}}^{p+q-1}\\
&\le \|u\|_{L^{m+q-1}(B_R)}^{\theta(p+q-1)}\,\|u\|_{L^{p+m+q-1}(B_R)}^{(1-\theta)(p+q-1)}\\
&\le \|u\|_{L^{m+q-1}(B_R)}^{\theta(p+q-1)}\,\left[\|u\|_{L^{p\frac N2}(B_R)}^{(1-\theta)\frac{p}{p+m+q-1}}\,\|u\|_{L^{(m+q-1)\frac {N}{N-2}}(B_R)}^{(1-\theta)\frac{m+q-1}{p+m+q-1}}\right]^{p+q-1}\\
&\le \|u\|_{L^{m+q-1}(B_R)}^{\theta(p+q-1)}\,\|u\|_{L^{p\frac N2}(B_R)}^{(1-\theta)\frac{p(p+q-1)}{p+m+q-1}}\,\left(\frac{1}{C_s}\left\|\nabla \left(u^{\frac{m+q-1}{2}}\right)\right\|_{L^2(B_R)}\right)^{2(1-\theta)\frac{p+q-1}{p+m+q-1}}
\end{aligned}
\end{equation}
where $\theta:=\frac{m(m+q-1)}{p(p+q-1)}$. By plugging \eqref{eq74} and \eqref{eq75} into \eqref{eq73} we obtain

\begin{equation}\label{eq75b}
\begin{aligned}
\frac{1}{q}\frac{d}{dt}\|u(t)\|_{L^q(B_R)}^{q}
& \le-\frac{4m(q-1)}{(m+q-1)^2}\, c_1\,C_p^2 \left\| u(t)\right\|^{m+q-1}_{L^{m+q-1}(B_R)}\, \\
& - \frac{4m(q-1)}{(m+q-1)^2}\, c_2\,C_p^{2\alpha} \left\| u(t)\right\|^{\alpha(m+q-1)}_{L^{m+q-1}(B_R)}\, \left\|\nabla \left(u^{\frac{m+q-1}{2}}\right)\right\|_{L^2(B_R)}^{2-2\alpha} \\
& +\tilde{C}\|u(t)\|_{L^{m+q-1}(B_R)}^{\theta(p+q-1)}\,\|u(t)\|_{L^{p\frac N2}(B_R)}^{(1-\theta)\frac{p(p+q-1)}{p+m+q-1}}\,\left\|\nabla \left( u^{\frac{m+q-1}{2}}\right)\right\|_{L^2(B_R)}^{2(1-\theta)\frac{p+q-1}{p+m+q-1}}
\end{aligned}
\end{equation}
where \begin{equation}\label{tildec}
\tilde{C}=\left(\frac{1}{C_s}\right)^{2(1-\theta)\frac{p+q-1}{p+m+q-1}}.
\end{equation} Let us now fix $\alpha\in (0,1)$ such that
$$
2-2\alpha=2(1-\theta)\left(\frac{p+q-1}{p+m+q-1}\right).
$$
Hence we have
\begin{equation}\label{eq76}
\alpha\,=\,\frac mp.
\end{equation}
By substituting \eqref{eq76} into \eqref{eq75b} we obtain
\begin{equation}\label{eq77}
\begin{aligned}
\frac{1}{q}\frac{d}{dt}\|u(t)\|_{L^q(B_R)}^{q} &\le -\frac{4m(q-1)}{(m+q-1)^2}\, c_1\,C_p^2 \left\| u(t)\right\|^{m+q-1}_{L^{m+q-1}(B_R)}\, \\
& -  \frac{1}{\tilde C}\left\{ \frac{4m(q-1)C}{(m+q-1)^2}\,  - \left\| u(t)\right\|^{\frac{p(p+q-1)-m(m+q-1)}{p+m+q-1}}_{L^{p\frac N2}(B_R)}\right\}  \\
&\times \left\| u(t)\right\|^{\alpha(m+q-1)}_{L^{m+q-1}(B_R)}\, \left\|\nabla \left(u^{\frac{m+q-1}{2}}\right)\right\|_{L^2(B_R)}^{2-2\alpha},
\end{aligned}
\end{equation}
where $C$ has been defined in Remark \ref{remark5}. Observe that, thanks to hypothesis \eqref{epsilon11a} and the continuity of the solution $u(t)$, there exists $t_0>0$ such that
$$
\left\| u(t)\right\|_{L^{p\frac N2}(B_R)}\le 2\, \tilde\varepsilon_1\,\,\,\,\,\text{for any}\,\,\,\, t\in (0,t_0]\,.
$$
Hence \eqref{eq77} becomes, for any $t\in (0,t_0]$
\begin{equation*}
\begin{aligned}
\frac{1}{q}\frac{d}{dt}\|u(t)\|_{L^q(B_R)}^{q} &\le -\frac{4m(q-1)}{(m+q-1)^2}\, c_1\,C_p^2 \left\| u(t)\right\|^{m+q-1}_{L^{m+q-1}(B_R)}\, \\
& -  \frac{1}{\tilde C}\left\{ \frac{4m(q-1)C}{(m+q-1)^2}\,  -2\tilde \varepsilon_1^{{\frac{p(p+q-1)-m(m+q-1)}{p+m+q-1}}}\right\}  \left\| u(t)\right\|^{\alpha(m+q-1)}_{L^{m+q-1}(B_R)}\, \left\|\nabla \left(u^{\frac{m+q-1}{2}}\right)\right\|_{L^2(B_R)}^{2-2\alpha}\\
&\le 0\,,
\end{aligned}
\end{equation*}
 provided $\tilde\varepsilon_1$ is small enough.
Hence we have proved that $\|u(t)\|_{L^q(B_R)}$ is decreasing in time for any $t\in (0,t_0]$, i.e.
\begin{equation}\label{eq78}
\|u(t)\|_{L^q(B_R)}\le \|u_0\|_{L^q(B_R)}\quad \text{for any} \,\,\,t\in (0,t_0]\,.
\end{equation}
In particular, inequality \eqref{eq78} holds $q=p\frac N2$. Hence we have
$$
\|u(t)\|_{L^{p\frac N2}(B_R)}\le \|u_0\|_{L^{p\frac N2}(B_R)}\,<\,\tilde\varepsilon_1\quad \text{for any} \,\,\,\,t\in (0,t_0]\,.
$$
Now, we can repeat the same argument in the time interval $(t_0, t_1]$ where $t_1$ is chosen, thanks to the continuity of $u(t)$, in such a way that
$$
\left\| u(t)\right\|
\le 2\, \tilde\varepsilon_1\,\,\,\,\,\text{for any}\,\,\, t\in (t_0,t_1]\,.
$$
Thus we get
\begin{equation*}
\|u(t)\|_{L^q(B_R)}\le \|u_0\|_{L^q(B_R)}\quad \text{for any} \,\,\,t\in (0,t_1]\,.
\end{equation*}
Iterating this procedure we obtain the thesis.

\end{proof}

\begin{proposition}\label{prop71}
Let $m>1$, $p>m$, $R>0$, $u_0\in  L^{\infty}(B_R)$, $u_0\ge 0$. Let
\begin{equation}\label{eq710}
r>\, \frac N2, \quad\quad\quad s=1+\frac 2N-\frac 1r.
\end{equation}
Suppose that \eqref{epsilon11} holds for $\varepsilon_1=\varepsilon_1(p, m, N, r, C_s, C_p)$ sufficiently small.
Let $u$ be the solution to problem \eqref{eq31}. Let $M$ support the Sobolev and Poincaré inequalities \eqref{S} and \eqref{P}.
Then there exists $\Gamma=\Gamma(N,m,l,C_s)>0$ independent of $T$ such that, for all $t>0$,
\begin{equation}\label{eq711}
\|u(t)\|_{L^{\infty}(B_R)}
\le \Gamma\, \|u_0\|_{L^{m}(B_R)}^{\frac{s-1}{s}}\left[ \|u_0\|_{L^{pr}(B_R)}^{p}+\frac{1}{(m-1)t}\|u_0\|_{L^{r}(B_R)}\right]^{\frac{1}{ms}}.
\end{equation}
\end{proposition}

\begin{remark}
\rm If in Proposition \ref{prop71}, in addition, we assume that for some $k_0>0$
$$
\max\left\{\|u_0\|_{L^m(B_R)};\,\,\|u_0\|_{L^{pr}(B_R)}; \,\,\|u_0\|_{L^{r}(B_R)}\right\}\leq k_0\quad \textrm{ for every }\,\, R>0\,,
$$
then the bound from above for $\|u(t)\|_{L^{\infty}(B_R)}$ in \eqref{eq711} is independent of $R$.
\end{remark}

\begin{proof}[Proof of Proposition \ref{prop71}]
Let us set $w=u(\cdot,t)$. Observe that $w^m\in H_0^1(B_R)$ and $w\ge0$. Due to Proposition \ref{prop43} we know that
\begin{equation*}
-\Delta(w^m) \le \left [w^p+\frac{w}{(m-1)t} \right].
\end{equation*}
Observe that, since $u_0\in L^{\infty}(B_R)$ also $w\in L^{\infty}(B_R)$.
Due to \eqref{eq710}, we can apply Proposition \ref{prop1}, so we have that
$$
\|w\|_{L^{\infty}(B_R)}^m\le \frac{s}{s-1}\left(\frac{1}{C_s}\right)^{\frac {2}{s}} \left\|w^p+\frac{w}{(m-1)t} \right\|_{L^{r}(B_R)}^{\frac 1s}\|w^m\|_{L^{1}(B_R)}^{\frac{s-1}{s}}
$$
Therefore
\begin{equation}\label{eq712}
\|w\|_{L^{\infty}(B_R)}^m\le \frac{s}{s-1} \left(\frac{1}{C_s}\right)^{\frac {2}{s}}\left\{\|w^p\|_{L^{r}(B_R)}+\frac{1}{(m-1)t} \|w\|_{L^{r}(B_R)}\right\}^{\frac 1s}\|w\|_{L^{m}(B_R)}^{m\frac{s-1}{s}},
\end{equation}
where $s$ has been defined in \eqref{eq710}. In view of \eqref{epsilon11} with a suitable $\varepsilon_1$, 
 since $u_0\in L^{\infty}(B_R)$, we can apply Lemma \ref{lemma71}.
Hence we obtain
\begin{equation}\label{eq713}
\|w^p\|_{L^{r}(B_R)} = \left\|w\right\|^p_{L^{pr}(B_R)} \le \|u_0\|_{L^{pr}(B_R)}^{p}.
\end{equation}
Similarly, again for an appropriate $\varepsilon_1$ in \eqref{epsilon11}, since $u_0\in L^{\infty}(B_R)$, we can apply Lemma \ref{lemma71}
and obtain
\begin{equation}\label{eq714}
\|w\|_{L^{r}(B_R)} \le \|u_0\|_{L^{r}(B_R)}.
\end{equation}
Plugging \eqref{eq713} and \eqref{eq714} into \eqref{eq712} we obtain
$$
\begin{aligned}
\|w\|^m_{L^{\infty}(B_R)} &\le \frac{s}{s-1}\left(\frac{1}{C_s}\right)^{\frac {2}{s}} \left\{\|w\|^p_{L^{pr}(B_R)}+\frac{1}{(m-1)t}\|w\|_{L^{r}(B_R)} \right\}^{\frac 1s} \|w\|_{L^{m}(B_R)}^{m\frac{s-1}{s}}\\
&\le\frac{s}{s-1}\left(\frac{1}{C_s}\right)^{\frac {2}{s}}\left\{\|u_0\|_{L^{pr}(B_R)}^{p}+\frac{1}{(m-1)t}\|u_0\|_{L^{r}(B_R)} \right\}^{\frac{1}{s}} \|w\|_{L^{m}(B_R)}^{m\frac{s-1}{s}}.
\end{aligned}
$$
Moreover, since $u_0\in L^{\infty}(B_R)$, we can apply Lemma \ref{lemma71} to $w$ with $q=m$. Thus from \eqref{eq72} with $q=m$ we get
\begin{equation}\label{eq715}
\|w\|_{L^{\infty}(B_R)}
\le \left[\frac{s}{s-1}\left(\frac{1}{C_s}\right)^{\frac {2}{s}}\right]^{\frac 1m}\|u_0\|_{L^{m}(B_R)}^{\frac{s-1}{s}}\left[ \|u_0\|_{L^{pr}(B_R)}^{p}+\frac{1}{(m-1)t}\|u_0\|_{L^{r}(B_R)}\right]^{\frac{1}{ms}}.
\end{equation}
We define
\begin{equation}\label{gamma}
\Gamma:=\left[\frac{s}{s-1}\left(\frac{1}{C_s}\right)^{\frac {2}{s}}\right]^{\frac 1m}.
\end{equation}
Then from \eqref{eq715} we get
$$
\|w\|_{L^{\infty}(B_R)}
\le \Gamma\|u_0\|_{L^{m}(B_R)}^{\frac{s-1}{s}}\left[ \|u_0\|_{L^{pr}(B_R)}^{p}+\frac{1}{(m-1)t}\|u_0\|_{L^{r}(B_R)}\right]^{\frac{1}{ms}}.
$$
\end{proof}

\begin{proof}[Proof of Theorem \ref{teo71}] The proof of Theorem \ref{teo71} follows the same line of arguments of that of Theorem \ref{teo22}, with minor differences. Let $\{u_{0,h}\}_{h\ge 0}$ be a family of functions such that
\begin{equation*}
\begin{aligned}
&(a)\,\,u_{0,h}\in L^{\infty}(M)\cap C_c^{\infty}(M) \,\,\,\text{for all} \,\,h\ge 0, \\
&(b)\,\,u_{0,h}\ge 0 \,\,\,\text{for all} \,\,h\ge 0, \\
&(c)\,\,u_{0, h_1}\leq u_{0, h_2}\,\,\,\text{for any } h_1<h_2,  \\
&(d)\,\,u_{0,h}\longrightarrow u_0 \,\,\, \text{in}\,\, L^{\theta}(M)\cap L^{pr}(M)\,\,\text{where}\,\,\theta:=\min\{m,r\}\quad \textrm{ as }\, h\to +\infty\,,\\
\end{aligned}
\end{equation*}
Observe that, due to assumptions $(c)$ and $(d)$, $u_{0,h}$ satisfies \eqref{epsilon11} for an appropriate $\varepsilon_1$ sufficiently small.
Moreover, thanks by interpolation, since $m<p<pr$, we have
$$
\,\,u_{0,h}\longrightarrow u_0 \,\,\, \text{in}\,\, L^p(M)\quad \textrm{ as }\, h\to +\infty\,.
$$
For any $R>0$, $k>0$, $h>0$, consider the problem
\begin{equation}\label{eq722}
\begin{cases}
u_t= \Delta u^m +T_k(u^p) &\text{in}\,\, B_R\times (0,+\infty)\\
u=0& \text{in}\,\, \partial B_R\times (0,\infty)\\
u=u_{0,h} &\text{in}\,\, B_R\times \{0\}\,. \\
\end{cases}
\end{equation}
From standard results it follows that problem \eqref{eq722} has a solution $u_{h,k}^R$ in the sense of Definition \ref{def31}; moreover, $u^R_{h,k}\in C\big([0, T]; L^q(B_R)\big)$ for any $q>1$. Hence, it satisfies the inequalities in Lemma \ref{lemma71} and in Proposition \ref{prop71}, i.e., for any $t\in(0,+\infty)$,
\begin{equation*}
\|u_{h,k}^R(t)\|_{L^m(B_R)}\,\le\, \|u_{0,h}\|_{L^m(B_R)};
\end{equation*}
\begin{equation*}
\|u_{h,k}^R(t)\|_{L^p(B_R)}\,\le\, \|u_{0,h}\|_{L^p(B_R)};
\end{equation*}
\begin{equation*}
\|u_{h,k}^R\|_{L^{\infty}(B_R)}
\le \Gamma\,\|u_{0,h}\|_{L^{m}(B_R)}^{\frac{s-1}{s}}\left[ \|u_{0,h}\|_{L^{pr}(B_R)}^{p}+\frac{1}{(m-1)t}\|u_{0,h}\|_{L^{r}(B_R)}\right]^{\frac{1}{ms}},
\end{equation*}
with $r$ and $s$ as in \eqref{eq710} and $\Gamma$ as in \eqref{gamma}. Arguing as in the proof of Theorem \eqref{eq22}, we can pass to the limit as $k\to +\infty, R\to +\infty, h\to \infty$ obtaining a function $u$, which satisfies
\begin{equation}\label{eq733}
\|u(t)\|_{L^m(M)}\leq  \|u_{0}\|_{L^m(M)},
\end{equation}
\begin{equation}\label{eq733b}
\|u(t)\|_{L^p(M)}\leq  \|u_{0}\|_{L^p(M)},
\end{equation}
and
\begin{equation}\label{eq734}
\|u\|_{L^{\infty}(M)}
\le \Gamma\,\|u_{0}\|_{L^{m}(M)}^{\frac{s-1}{s}}\left[ \|u_{0}\|_{L^{pr}(M)}^{p}+\frac{1}{(m-1)t}\|u_{0}\|_{L^{r}(M)}\right]^{\frac{1}{ms}},
\end{equation}
with $r$ and $s$ as in \eqref{eq710} and $\Gamma$ as in \eqref{gamma}.
Moreover, for any $\varphi \in C_c^{\infty}(M\times(0,T))$ such that $\varphi(x,T)=0$ for any $x\in M$, the function $u$ satisfies
\begin{equation}\label{eq735}
\begin{aligned}
-\int_0^T\int_{M} u\,\varphi_t\,d\mu\,dt =&\int_0^T\int_{M} u^m\,\Delta\varphi\,d\mu\,dt+ \int_0^T\int_{M} u^p\,\varphi\,d\mu\,dt \\
& +\int_{M} u_{0}(x)\,\varphi(x,0)\,d\mu.
\end{aligned}
\end{equation}
Observe that, due to inequalities \eqref{eq733}, \eqref{eq733b} and \eqref{eq734}, all the integrals in \eqref{eq735} are finite, hence $u$ is a solution to problem \eqref{problema} in the sense of Definition \ref{def21}.  Finally,  using hypothesis \eqref{epsilon2a}, inequality \eqref{eq721} can be derived exactly as \eqref{eq22}.
\end{proof}

\section{Proofs of Theorems \ref{teo24} and \ref{teo71W}}\label{weight}
We use the following Aronson-Benilan type estimate (see \cite{AB}; see also \cite[Proposition 2.3]{Sacks}); it can be shown exactly as Proposition \ref{prop43}.
\begin{proposition}\label{prop62}
Let $m>1$, $p>m$, $u_0\in H_0^1(B_R) \cap L^{\infty}(B_R)$, $u_0\ge 0$. Let $u$ be the solution to problem \eqref{eq61}. Then, for a.e. $t\in(0,T)$,
\begin{equation*}
-\Delta u^m(\cdot,t) \le \rho u^p(\cdot, t)+ \frac{\rho}{(m-1)t} u(\cdot,t) \quad \text{in}\,\,\,\mathfrak{D}'(B_R).
\end{equation*}
\end{proposition}

\smallskip

For any $R>0$, consider the following approximate problem
\begin{equation}\label{eq61}
\begin{cases}
\, \rho(x) u_t= \Delta u^m +\, \rho(x) u^p & \text{in}\,\, B_R\times (0,T) \\
u=0 &\text{in}\,\, \partial B_R\times (0,T)\\
u =u_0 &\text{in}\,\, B_R\times \{0\}\,,
\end{cases}
\end{equation}
where $B_R$ denotes the Euclidean ball with radius $R$ and centre in the origin $O$.

We exploit the following estimate, which can be proved as that in Lemma \ref{lemma41}.
\begin{lemma}\label{lemma63}
Let
$$m>1,\quad\quad p>m+\frac{2}{N}.$$
Suppose that inequality \eqref{S-pesi} holds. Suppose that $u_0\in L^{\infty}(B_R)$, $u_0\ge0$. Let $1<q<\infty$, $p_0$ be as in \eqref{p0} and assume that
\begin{equation*}
\|u_0\|_{\textrm L^{p_0}(B_R)}\,<\,\bar\varepsilon,
\end{equation*}
for $\bar\varepsilon=\bar\varepsilon(p, m, C_s, q)$ small enough. Let $u$ be the solution of problem \eqref{eq61}. Then
\begin{equation*}
\|u(t)\|_{L^q_{\rho}(B_R)} \le \|u_0\|_{L^q_{\rho}(B_R)}\quad \textrm{ for all }\,\, t>0\,.
\end{equation*}
\end{lemma}

The following smoothing estimate is also used; the proof is the same as that of Proposition \ref{prop42}.
\begin{proposition}\label{prop61}
Let
$$m>1,\quad\quad p>m+\frac{2}{N},$$
Assume \eqref{rho2} and \eqref{S-pesi}. Suppose that $u_0\in L^{\infty}(B_R)$, $u_0\ge0$. Let $u$ be the solution of problem \eqref{eq61}. 
Assume that \eqref{epsilon0} holds for $\varepsilon_0=\varepsilon_0(p, m, N, r, C_s)$ sufficiently small. Than there exists $C(m,q_0,C_s, \varepsilon, N, q)>0$ such that
\begin{equation*}
\|u(t)\|_{L^q_{\rho}(B_R)} \le C\,t^{-\gamma_q}\|u_0\|^{\delta_q}_{L^{q_0}_{\rho}(B_R)}\quad \textrm{ for all }\,\, t>0\,,
\end{equation*}
where
\begin{equation*}
\gamma_q=\left(\frac{1}{q_0}-\frac{1}{q}\right)\frac{N\,q_0}{2\,q_0+N(m-1)}\,;\quad \delta_q=\frac{q_0}{q}\left(\frac{q+\frac{N}{2}(m-1)}{q_0+\frac{N}{2}(m-1)}\right)\,.
\end{equation*}
\end{proposition}

\begin{proof}[Proof of Theorem \ref{teo24}] The conclusion follows by repeating the same arguments as in the proof of Theorem \ref{teo22}.  We use Lemma \ref{lemma63} instead of Lemma \ref{lemma41}, Proposition \ref{prop61} instead of \ref{prop42} and Proposition \ref{prop62} instead of Proposition \ref{prop43}.
\end{proof}

\subsection{Proof of Theorem \ref{teo71W}}
We consider problem \eqref{eq61}. We use the following estimate, which can be proved as that in Lemma \ref{lemma71}.

\begin{lemma}\label{lemma71W}
Let
$$m>1,\quad\quad p>m.$$
Assume that \eqref{S-pesi} and \eqref{P-pesi} hold. Suppose that $u_0\in L^{\infty}(B_R)$, $u_0\ge0$. Let $1<q<\infty$ and assume that
and assume that
\begin{equation*}
\|u_0\|_{L^{p\frac N2}(B_R)}<\tilde\varepsilon_1
\end{equation*}
for a suitable $\tilde\varepsilon_1=\tilde \varepsilon_1(p, m, N, C_p, C_s, q)$ sufficiently small.  Let $u$ be the solution of problem \eqref{eq61}. Then
\begin{equation*}
\|u(t)\|_{L^q(B_R)} \le \|u_0\|_{L^q(B_R)}\quad \textrm{ for all }\,\, t>0\,.
\end{equation*}
\end{lemma}

\begin{proof}[Proof of Theorem \ref{teo71W}] The conclusion follows arguing step by step as in the proof of Theorem \ref{teo71}.  We use Lemma \ref{lemma71W} instead of Lemma \ref{lemma71} and Proposition \ref{prop62} instead of Proposition \ref{prop43}.
\end{proof}

\bigskip

\noindent \bf Acknowledgments. \rm The first and third author are partially supported by the PRIN project
201758MTR2 “Direct and inverse problems for partial differential equations: theoretical aspects
and applications” (Italy). All authors are members of the Gruppo Nazionale per l’Analisi Mate\-ma\-tica,
la Probabilità e le loro Applicazioni (GNAMPA) of the Istituto Nazionale di Alta
Mate\-ma\-tica (INdAM). The third author is partially supported by GNAMPA Projects 2019,
2020.

%
%


\bigskip
\bigskip
\bigskip




\begin{thebibliography}{5}
\bibitem{A} D. Alikakos, \emph{$L^p$ bounds of solutions of reaction-diffusion equations}, Comm. Partial Differential Equations {\bf 4} (1979), 827--868\,.

\bibitem{AB} D. G. Aronson, P. Bénilan, \emph{Regularité des solutions de l'éequation des milieux poreus dans $\mathbb R^N$}, C. R. Acad. Sci. Paris Ser. A-B {\bf 288} (1979), 103--105\,.

\bibitem{ACP} D. Aronson, M.G. Crandall, L.A. Peletier, \emph{Stabilization of solutions of a degenerate nonlinear diffusion problem}, Nonlinear Anal. \bf 6 \rm (1982), 1001--1022.


\bibitem{BC} L. Boccardo, G. Croce, "Elliptic partial differential equations. Existence and regularity of distributional solutions", De Gruyter, Studies in Mathematics, 55, 2013\,.

\bibitem{BPT} C. Bandle, M.A. Pozio, A. Tesei, \em The Fujita exponent for the Cauchy problem in the hyperbolic space\rm,
J. Differential Equations \bf 251 \rm (2011), 2143--2163.

\bibitem{BG} M. Bonforte, G. Grillo, \em Asymptotics of the porous media
equations via Sobolev inequalities\rm, J. Funct.
Anal. {\bf 225} (2005), 33-62.


\bibitem{CFG} X. Chen, M. Fila, J.S. Guo, \emph{Boundedness of global solutions of a supercritical parabolic equation}, Nonlinear Anal. {\bf 68} (2008), 621--628.



\bibitem{F} H. Fujita, \em On the blowing up of solutions of the Cauchy problem for $u_t=\Delta u+u^{1+\alpha}$\rm, J. Fac. Sci. Univ. Tokyo Sect. I \textbf{13} (1966), 109--124.

\bibitem{FI}  Y. Fujishima, K. Ishige, \em Blow-up set for type I blowing up solutions for a semilinear heat equation\rm, Ann. Inst. H. Poincaré Anal. Non Lin\'eaire \bf 31 \rm (2014), 231--247.
\bibitem{GV} V.A. Galaktionov, J.L. V\'azquez, \em Continuation of blowup solutions of nonlinear heat equations in several dimensions\rm, Comm. Pure Appl. Math. \bf 50 \rm (1997), 1--67.

\bibitem{Grig} A. Grigor'yan, \emph{Analytic and geometric background of recurrence and non-explosion of the Brownian motion on Riemannian manifolds}, Bull.~Amer.~Math.~Soc. {\bf 36} (1999), 135--249.

\bibitem{Grig3} A. Grigor'yan, ``Heat Kernel and Analysis on Manifolds'', AMS/IP Studies in Advanced Mathematics, 47, American Mathematical Society, Providence, RI; International Press, Boston, MA, 2009.

\bibitem{GIM} G. Grillo, K. Ishige, M. Muratori, \emph{Nonlinear characterizations of stochastic completeness}\rm, J. Math. Pures Appl. \bf 139 \rm (2020), 63-82.

\bibitem{GMeP} G.Grillo, G. Meglioli, F. Punzo, \emph{Smoothing effects and infinite time blowup for reaction-diffusion equations: an approach via Sobolev and Poincar\'{e} inequalities}\rm, preprint (2020)

\bibitem{GMhyp} G. Grillo, M. Muratori, \emph{Radial fast diffusion on the hyperbolic space}, Proc.~Lond.~Math.~Soc. {\bf 109} (2014), 283--317.

\bibitem{GM2} G. Grillo, M. Muratori, \emph{Smoothing effects for the porous medium equation on Cartan-Hadamard manifolds}, Nonlinear Anal. {\bf 131} (2016), 346--362.

\bibitem{GMPo} G. Grillo, M. Muratori, M.M. Porzio, \em Porous media equations with two weights: smoothing and decay properties of energy solutions via Poincaré inequalities\rm,
Discrete Contin. Dyn. Syst. \bf 33 \rm (2013), 3599–3640.

\bibitem{GMPbd} G. Grillo, M. Muratori, F. Punzo, {\em The porous medium equation with large initial data on negatively curved Riemannian manifolds}, J. Math. Pures Appl. \bf 113 \rm (2018), 195--226.

\bibitem{GMPrm} G. Grillo, M. Muratori, F. Punzo, \emph{The porous medium equation with measure data on negatively curved Riemannian manifolds},  J. European Math. Soc. \bf 20 \rm (2018), 2769-2812.

    \bibitem{GMPv} G. Grillo, M. Muratori, F. Punzo, \emph Blow-up and global existence for the porous medium equation with reaction on a class of Cartan-Hadamard manifolds\rm, J. Diff. Eq. \bf 266 \rm (2019), 4305-4336.


\bibitem{GMV} G. Grillo, M. Muratori, J.L. V\'azquez, \emph{The porous medium equation on Riemannian manifolds with negative curvature. The large-time behaviour}, Adv. Math. \bf 314 \rm (2017), 328--377.

\bibitem{H} K. Hayakawa, \em On nonexistence of global solutions of some semilinear parabolic differential equations\rm, Proc. Japan Acad. \textbf{49} (1973), 503--505.

   \bibitem{KR} S. Kamin, P. Rosenau, \em Nonlinear thermal evolution in an inhomogeneous medium\rm, J. Math. Phys. \bf 23 \rm (1982), 1385–1390.

\bibitem{KS} D. Kinderlehrer, G. Stampacchia, "An Introduction to Variational Inequalities and Their Applications", Academic Press, New York, 1980.

    \bibitem{L} H.A. Levine, \em The role of critical exponents in blow-up theorems\rm, SIAM Rev. \bf 32 \rm (1990), 262--288.

\bibitem{MT} A.V. Martynenko, A. F. Tedeev, \em On the behavior of solutions of the Cauchy problem for a degenerate parabolic equation with nonhomogeneous density and a source, \rm (Russian) Zh. Vychisl. Mat. Mat. Fiz. \bf 48  \rm (2008), no. 7, 1214-1229; transl. in Comput. Math. Math. Phys. \bf 48 \rm (2008), no. 7, 1145-1160.

\bibitem{MTS} A.V. Martynenko, A.F. Tedeev, V.N. Shramenko, \em The Cauchy problem for a degenerate parabolic equation with inhomogenous density and a source in the class of slowly vanishing initial functions \rm  (Russian) Izv. Ross. Akad. Nauk Ser. Mat. \bf 76 \rm (2012), no. 3, 139-156; transl. in Izv. Math. \bf 76 \rm (2012), no. 3, 563-580.

\bibitem{MTS2} A.V. Martynenko, A.F. Tedeev, V.N. Shramenko, \em On the behavior of solutions of the Cauchy problem for a degenerate parabolic equation with source in the case
where the initial function slowly vanishes\rm, Ukrainian Math. J. \bf 64 \rm (2013), 1698–1715.


\bibitem{MMP} P. Mastrolia, D. D. Monticelli, F. Punzo, \em{Nonexistence of solutions to parabolic differential inequalities with a potential on Riemannian manifolds}, \rm Math. Ann. 367 (2017), 929-963.

\bibitem{MP1} G. Meglioli, F. Punzo, \em Blow-up and global existence for solutions to the porous medium equation with reaction and slowly decaying density, \rm  J. Diff. Eq., 269 (2020), 8918-8958.

\bibitem{MP2} G. Meglioli, F. Punzo, \em Blow-up and global existence for solutions to the porous medium equation with reaction and fast decaying density, \rm Nonlin. Anal. 203 (2021), 112187.

    \bibitem{MQV} N. Mizoguchi, F. Quir\'os, J.L. V\'azquez, \em Multiple blow-up for a porous medium equation with reaction\rm,
Math. Ann. \bf 350 \rm (2011), 801--827.

    \bibitem{Pu1} F.  Punzo, {\it Support  properties  of  solutions  to nonlinear  parabolic  equations  with variable density in the hyperbolic space}, Discrete Contin. Dyn. Syst. Ser. S \textbf{5} (2012), 657--670.

        \bibitem{Pu3} F. Punzo, {\it Blow-up of solutions to semilinear parabolic equations on Riemannian manifolds with negative sectional curvature}, J. Math. Anal. Appl. {\bf 387} (2012), 815--827.

   \bibitem{Q} P. Quittner, {\it The decay of global solutions of a semilinear heat equation}, Discrete Contin. Dyn. Syst. {\bf 21} (2008), 307--318.

\bibitem{S} P. Souplet, \emph{Morrey spaces and classification of global solutions for a supercritical semilinear heat equation in
$\mathbb R^n$}, J. Funct. Anal. {\bf 272} (2017), 2005--2037.

\bibitem{Sacks} P.E. Sacks, \emph{Global behavior for a class of nonlinear evolution equations}, SIAM J. Math Anal. \bf 16 \rm (1985), 233--250.

 \bibitem{SGKM}  A.A. Samarskii, V.A. Galaktionov, S.P. Kurdyumov, A.P. Mikhailov, ``Blow-up in Quasilinear Parabolic Equations'', De Gruyter Expositions in Mathematics, 19. Walter de Gruyter \& Co., Berlin, 1995.

\bibitem{Vaz1} J.L. V\'azquez, {\it The problems of blow-up for nonlinear heat equations. Complete blow-up and avalanche formation},  Atti Accad. Naz. Lincei Cl. Sci. Fis. Mat. Natur. Rend. Lincei Mat. Appl. \textbf{15} (2004), 281--300.


\bibitem{V} J.L. V\'azquez, ``The Porous Medium Equation. Mathematical Theory'', Oxford Mathematical Monographs. The Clarendon Press, Oxford University Press, Oxford, 2007.

\bibitem{VazH} J.L. V\'azquez, {\it Fundamental solution and long time behavior of the porous medium equation in hyperbolic space}, J. Math. Pures Appl. {\bf 104} (2015), 454--484.

\bibitem{WY} Z. Wang, J. Yin, {\it A note on semilinear heat equation in hyperbolic space}, J. Differential Equations {\bf 256} (2014), 1151--1156.

\bibitem{WY2} Z. Wang, J. Yin, {\it Asymptotic behaviour of the lifespan of solutions for a semilinear heat equation in hyperbolic space}, Proc. Roy. Soc. Edinburgh Sect. A {\bf 146} (2016) 1091--1114.

\bibitem{W} F.B. Weissler, \em $L^p$-energy and blow-up for a semilinear heat equation\rm, Proc. Sympos. Pure Math. \bf 45 \rm (1986), 545--551.

\bibitem{Y} E. Yanagida, \emph{Behavior of global solutions of the Fujita equation}, Sugaku Expositions {\bf 26} (2013), 129--147.

\bibitem{Z} Q.S. Zhang, \em Blow-up results for nonlinear parabolic equations on manifolds\rm, Duke Math. J. \bf 97 \rm (1999), 515--539.

\end{thebibliography}
%


\end{document}